\documentclass{amsart}
\usepackage{a4wide}
\usepackage[utf8]{inputenc}
\usepackage{amssymb}
\usepackage{amsmath}
\usepackage{mathrsfs} 
\usepackage{amsthm}
\usepackage{mathtools}
\usepackage{leftidx} 
\usepackage{color}
\usepackage{enumitem}

\usepackage{caption}
\usepackage{tikz}
\usepackage{xcolor}
\usepackage{upgreek}
\usepackage{hyperref}
\usepackage{bbm}

\newcommand{\N}{\mathbb{N}}
\newcommand{\R}{\mathbb{R}}
\newcommand{\C}{\mathbb{C}}
\newcommand{\K}{\mathbb{K}}
\renewcommand{\P}{\mathbb{P}}
\newcommand{\E}{\mathbb{E}}
\newcommand{\PP}{\mathbb{P}}
\newcommand{\ce}{\coloneqq}
\newcommand{\ec}{\eqqcolon}
\newcommand{\1}{\mathbf{1}}

\newcommand{\calB}{\mathcal{B}}

\newcommand{\calF}{\mathcal{F}}

\newcommand{\calL}{\mathcal{L}}

\newcommand{\calP}{\mathcal{P}}

\newcommand{\filtrF}{\mathscr{F}}

\newcommand{\rmd}{\mathrm{d}}

\newcommand{\hra}{\hookrightarrow}

\DeclareMathOperator{\TextRe}{Re}

\newcommand{\seq}{\subseteq}

\newcommand{\ve}{\varepsilon}
\newcommand{\vp}{\varphi}

\newcommand{\ee}{\mathrm{e}}
\newcommand{\from}{\colon}

\newcommand{\n}{\|}
\newcommand{\one}{{\bf 1}}

\renewcommand{\ss}{^\star}

\newcommand{\nn}{|\!|\!|}

\newcommand{\Om}{\Omega}
\newcommand{\om}{\omega}
\newcommand{\F}{\mathcal{F}}

\newcommand{\gHX}{{\gamma(H,X)}}

\newcommand{\CF}{C_{F}}
\newcommand{\CG}{C_{G}}

\newcommand{\ds}{\,\mathrm{d}s}
\newcommand{\dWHs}{\,\mathrm{d}W_H(s)}
\newcommand{\LF}{L_{F}}
\newcommand{\LG}{L_{G}}
\newcommand{\LFm}{L_{F,m}}
\newcommand{\LGm}{L_{G,m}}

\newcommand{\iu}{\mathrm{i}} 
\newcommand{\gHY}{{\gamma(H,Y)}}

\newcommand{\li}{\leftidx} 
\newcommand{\mU}{\leftidx{_m}{U}} 
\newcommand{\muo}{\li{_m}u_0}
\newcommand{\mF}{\li{_m}F}
\newcommand{\mG}{\li{_m}G}

\newtheorem{Satz}{Satz}[section]
\newtheorem{definition}[Satz]{Definition} 
\newtheorem{theorem}[Satz]{Theorem}
\newtheorem{lemma}[Satz]{Lemma}	
\newtheorem{proposition}[Satz]{Proposition}
\newtheorem{corollary}[Satz]{Corollary}
\newtheorem{remark}[Satz]{Remark}

\newtheorem{assumption}[Satz]{Assumption}

\allowdisplaybreaks
\numberwithin{equation}{section}

\begin{document}

\title[Temporal approximation of SPDEs with irregular nonlinearities]{Temporal approximation of stochastic evolution equations with irregular nonlinearities}
\author{Katharina Klioba}\address{
Hamburg University of Technology,
Institute of Mathematics,  D-21073 Hamburg, Germany}
\email{Katharina.Klioba@tuhh.de}

\author{Mark Veraar}
\address{Delft Institute of Applied Mathematics\\
Delft University of Technology \\ P.O. Box 5031\\ 2600 GA Delft\\The
Netherlands} \email{M.C.Veraar@tudelft.nl}

\date{\today}

\thanks{The second author is supported by the VICI subsidy VI.C.212.027 of the Netherlands Organisation for Scientific Research (NWO)}



\begin{abstract}
In this paper, we prove convergence for contractive time discretisation schemes for semi-linear stochastic evolution equations with irregular Lipschitz nonlinearities, initial values, and additive or multiplicative Gaussian noise on $2$-smooth Banach spaces $X$. The leading operator $A$ is assumed to generate a strongly continuous semigroup $S$ on $X$, and the focus is on non-parabolic problems. The main result concerns convergence of the {\em uniform strong error}
\[{\rm E}_{k}^{\infty} \coloneqq \Big(\E \sup_{j\in \{0, \ldots, N_k\}} \|U(t_j) - U^j\|_X^p\Big)^{1/p} \to 0\quad (k \to 0),\]
where $p \in [2,\infty)$, $U$ is the mild solution, $U^j$ is obtained from a time discretisation scheme, $k$ is the step size, and $N_k = T/k$ for final time $T>0$. This generalises previous results to a larger class of admissible nonlinearities and noise, as well as rough initial data from the Hilbert space case to more general spaces. We present a proof based on a regularisation argument. Within this scope, we extend previous quantified convergence results for more regular nonlinearity and noise from Hilbert to $2$-smooth Banach spaces. The uniform strong error cannot be estimated in terms of the simpler {\em pointwise strong error}
\[{\rm E}_k \coloneqq \bigg(\sup_{j\in \{0,\ldots,N_k\}}\E \|U(t_j) - U^{j}\|_X^p\bigg)^{1/p},\]
which most of the existing literature is concerned with.
Our results are illustrated for a variant of the Schrödinger equation, for which previous convergence results were not applicable. 
\end{abstract}

\maketitle

\section{Introduction}

This paper is concerned with the temporal discretisation of nonlinear stochastic PDEs driven by multiplicative Gaussian noise. We aim at proving convergence of time discretisation schemes for such equations, which can abstractly be written as a stochastic evolution equation of the form
\begin{align}
\label{eq:stEvolEqnIntro}
    \Bigg\{\begin{split} \rmd U &=(A U + F(U))\,\rmd t  + G(U) \,\rmd W_H~~~\text{ on } [0,T],\\ U(0) &= u_0 \in L^p(\Omega;X).
\end{split}
\end{align}
Here, $A$ generates a $C_0$-semigroup $(S(t))_{t \ge 0}$ on a $2$-smooth Banach space $X$, $W_H$ is a cylindrical Brownian motion, $T>0$, $p \in [2,\infty)$, and $u_0$ is the initial data. Besides global Lipschitz continuity, no further regularity assumptions are imposed on the nonlinearity $F$ and noise $G$.

Now, our goal is to show pathwise uniform convergence of contractive time discretisation schemes for such irregular nonlinearities and rough initial data, focusing on the hyperbolic setting. It has been extensively studied in recent years (see \cite{AC18, ACLW16, BLM21, berg2020exponential, BC23, CaiCohenWang2023FullWave, CCHS20, CL22, CLS13, CQS16, cox2019weak, Cui21, CuiHong, HM19, HHS22, JdNJW21, KliobaVeraar23, KLP20, KLL12, KLL13, KLS10, WGT14} and references therein). When passing to the parabolic setting (i.e., $(S(t))_{t\geq 0}$ being an analytic semigroup), regularisation phenomena allow for different proof techniques, resulting in much stronger convergence results. For details on the parabolic case, we refer to \cite{ACQ20, BL13, BeJe19, BHRR, CoxNee13, GyMi09, JK09, JK11, KamBlo, KLL15, KLL18, Kruse, LPS14} and references therein.  

\subsection{The setting}

The error usually considered in the above-mentioned literature on the hyperbolic case (as well as in the parabolic case) is the {\em pointwise strong error}
\begin{align}\label{eq:pointwiseerrorIntro}
    \sup_{j\in \{0,\ldots, N_k\}}\E\|U(t_j) - U^j\|^p,
\end{align}
where $U$ denotes the mild solution to \eqref{eq:stEvolEqnIntro} and $(U^j)_{j=0}^{N_k}$ an approximation thereof. The latter is recursively determined by $U^0 = u_0$,
\begin{equation}
\label{eq:UjschemeIntro}
    U^{j} = R_k U^{j-1} + k R_k F(U^{j-1})+ R_k G(U^{j-1}) \Delta W^{j}, \ \ j=1, \ldots, N_k,
\end{equation}
for some time discretisation scheme $R_k \in \calL(X)$ that is an approximation of the semigroup $S$ at time $k>0$. Here, $N_k = T/k$ is the number of time grid points, $k = t_j-t_{j-1}$ is the uniform step size, $t_j=jk$, and $\Delta W^{j} = W_H(t_j) - W_H(t_{j-1})$. 

It is a natural question to find convergence rates for
the {\em uniform strong error}
\begin{align}\label{eq:uniformerrorIntro}
    \E \sup_{j\in \{0, \ldots, N_k\}} \|U(t_j) - U^j\|^p,
\end{align}
since it describes moments of the maximal error in time rather than the maximum in time of moments of the error. It can, however, not be controlled by the simpler pointwise strong error \eqref{eq:pointwiseerrorIntro} without a loss in the rate. Indeed, if the pointwise strong error converges at rate $\alpha>0$, i.e., \eqref{eq:pointwiseerrorIntro} is bounded by $CN^{-\alpha p}$ for some $C>0$, then
\begin{align}
\label{eq:easyEstPointwiseStrong}
    \E \sup_{j \in \{1,\ldots,N\}} \|U(t_j)-U^j\|^p &\leq \E \sum_{j=1}^N \|U(t_j)-U^j\|^p = \sum_{j=1}^N \E  \|U(t_j)-U^j\|^p \nonumber\\
    &\leq N \sup_{j \in \{1,\ldots,N\}}  \E  \|U(t_j)-U^j\|^p \leq CN^{1-\alpha p}.
\end{align}
Taking $p$-th roots, convergence at rate $\alpha-\frac{1}{p}$ is obtained.
Since we have arbitrarily slow rates $\alpha \in (0,\frac{1}{2}]$ and are also interested in the case $p=2$, this estimate is not strong enough for our purposes. Still, convergence of the whole path can be expected, as numerical simulations suggest \cite{AC18, cohenMaxwell,  Wang15}.
Estimates where the supremum is inside the expectation are usually referred to as {\em maximal estimates}, and ample literature is available on this for general stochastic processes \cite{Talagrand}.

In \cite{KliobaVeraar23}, convergence rates for \eqref{eq:uniformerrorIntro} were obtained for general contractive time discretisation schemes under structural assumptions on the nonlinearity $F$ and the noise $G$ in \eqref{eq:stEvolEqnIntro}. Namely, it was assumed that they preserve the spatial regularity of the argument in the sense of the mapping properties $F:Y \to Y$, $G: Y \to \calL_2(H,Y)$ for a proper subspace $Y \hra X$ with additional smoothness, where $\calL_2(H,Y)$ denotes the space of Hilbert-Schmidt operators. Moreover, $F$ and $G$ were assumed to be of linear growth on $Y$ and the initial data to be of the same additional regularity. Naturally, the question arises whether these assumptions can be relaxed. Clearly, we cannot expect to preserve the convergence rate because this already fails in the linear deterministic case. However, it is an open question whether qualitatively, pathwise uniform convergence holds under weaker assumptions on $F$ and $G$ as well as for rough initial data from $L^p(\Omega;X)$. 

The main goal of our paper is to prove that this question can be answered positively, merely assuming global Lipschitz continuity of $F$ and $G$ on the full space $X$. In addition, we show that under the respective conditions, both convergence results with and without rate extend to the more general setting of $2$-smooth Banach spaces (cf. Subsection \ref{subsec:2smoothBanach}) rather than Hilbert spaces.

\subsection{Main result}

Before we can state our main result, we require an additional definition. Let $X$ and $Y$ be $2$-smooth Banach spaces such that $Y \hra X$. For $\alpha\in (0,1]$ we say that a time discretisation scheme $R:[0,T] \to \calL(X), k \mapsto R_k$ {\em approximates $S$ to order $\alpha$ on $Y$} if there is a constant $C_{\alpha} \ge 0$ such that for all $x \in Y$, $k>0$, and $j\in \{0,\ldots,N_k\}$
\begin{equation*}
    \|(S(t_j)-R_k^j)x\|_X \le C_\alpha k^\alpha\|x\|_Y.
\end{equation*}
Such a time discretisation scheme is called {\em contractive} if $\|R_k\|_{\calL(X)} \le 1$.
We denote by $\gamma(H,X)$ the space of $\gamma$-radonifying operators from a Hilbert space $H$ to $X$ (cf. Subsection \ref{subsec:radonifying}), which coincides with the space of Hilbert-Schmidt operators if $X$ is Hilbert. Our main theorem is as follows. 

\begin{theorem}
\label{thm:convergenceIntro}
    Let $X$ be a $2$-smooth Banach space. Suppose that $A$ generates a $C_0$-contraction semigroup $(S(t))_{t \ge 0}$ on both $X$ and $D(A)$. Let $(R_k)_{k>0}$ be a contractive time discretisation scheme on both $X$ and $D(A)$ that approximates $S$ to some order $\alpha \in (0,1]$ on $D(A)$. Suppose that $F\from X \to X$ and $G\from  X \to \gamma(H,X)$ are Lipschitz continuous. Let $T>0$, $p\in [2,\infty)$, and $u_0 \in L^p(\Omega;X)$. Denote by $U$ the mild solution to \eqref{eq:stEvolEqnIntro} on $[0,T]$. Let $k \in (0,T/2]$ and $(U^j)_{j=0,\ldots,N_k}$ be given by \eqref{eq:UjschemeIntro}. Define the piecewise constant extension $\tilde{U}\from [0,T] \to L^p(\Omega;X)$ by $\tilde{U}(t) \ce U^j$ for $t \in [t_j,t_{j+1})$, $0 \le j \le N_k-1$, and $\tilde{U}(T) \ce U^{N_k}$. Then
    \begin{align*}
        \lim_{k \to 0} \bigg\| \sup_{t \in [0,T]} \|U(t)-\tilde{U}(t)\|_{X}\bigg\|_{L^p(\Omega)} =0.
    \end{align*}
\end{theorem}

Theorem \ref{thm:convergenceIntro} follows from Theorem \ref{thm:convergence}. The pathwise uniform convergence rates in the structured setting with additional regularity, which are required in the proof of the main result, can be found in Theorem \ref{thm:convergenceRate}. 

Among the schemes covered by Theorem \ref{thm:convergenceIntro} and Theorem \ref{thm:convergenceRate} are
\begin{itemize}
\item exponential Euler (EE): $R_k = S(k)$;
\item implicit Euler (IE): $R_k = (1-kA)^{-1}$;
\item Crank-Nicolson (CN): $R_k = (2+kA)(2-kA)^{-1}$.
\end{itemize}
Contractivity of the scheme $R$ on $X$ and $D(A)$ is an immediate consequence of the contractivity of the semigroup $(S(t))_{t \ge 0}$ in the cases (EE) and (IE). For (CN) and other commonly used schemes, an argument based on functional calculus yields the desired contractivity. Applications to the Schr\"odinger equation are contained in the main paper (see Section \ref{sec:example}), showing pathwise uniform convergence for general contractive schemes in a setting in which previous results on (EE) \cite[Thm.~4.3]{AC18} or other contractive schemes \cite[Thm.~6.13]{KliobaVeraar23} are not applicable.

Naturally, in addition to the temporal discretisation investigated here, a space discretisation is needed for a numerical solution of \eqref{eq:stEvolEqnIntro}. In this paper, the focus lies on the temporal discretisation in the global Lipschitz setting, since a detailed understanding thereof is a quintessential step towards the treatment of locally Lipschitz nonlinearities. Our result should be seen as a first step, and we plan to continue our work on uniform strong errors in a locally Lipschitz setting in the near future.

\subsection{Method of proof}

Previous results on pathwise uniform convergence are only applicable if nonlinearity and noise preserve additional regularity, are of linear growth on the space with additional regularity, and the initial data are pathwise more regular as well. To circumvent the problem that this is not the case in our setting, we regularise the nonlinearity, the noise, and the initial values occurring in \eqref{eq:stEvolEqnIntro} according to 
\begin{equation*}
    \mF \ce mR(m,A)F,\quad \mG \ce mR(m,A)G,\quad \muo \ce mR(m,A)u_0
\end{equation*}
for some regularisation parameter $m \in \N$, where $R(m,A) \ce (m-A)^{-1}$ denotes the resolvent. By construction, $\mF$ maps to $D(A)$, $\mG$ maps to $\gamma(H,D(A))$, and $\muo \in L^p(\Omega;D(A))$, giving the desired additional regularity in structure with the more regular space being $D(A)$. Hence, this enables us to apply existing convergence rate results for the regularised discretisation given by
\begin{equation}
\label{eq:regDiscretisationIntro}
    \mU^j \ce R_k \mU^{j-1}+ k R_k \mF(\mU^{j-1}) + R_k \mG(\mU^{j-1})\Delta W_j,\quad \mU^0 \ce \muo,
\end{equation}
for $1 \le j \le N_k$. They approximate the mild solution $\mU$ of the regularised evolution equation
\begin{equation}
\label{eq:regStEvolEqnIntro}
    \mU = (A\mU + \mF(\mU))\,\rmd t + \mG(\mU)\,\rmd W_H(t),\quad\mU(0)=\muo \in X.
\end{equation}

Since the equations considered are in $2$-smooth Banach spaces, the results from \cite{KliobaVeraar23} have to be generalised beyond the Hilbert space setting. While most of the extension is straightforward, stability of the scheme can no longer be obtained by a dilation argument. Instead, the key ingredient of the proof is a novel maximal inequality for discrete convolutions based on a martingale argument. Lemma \ref{lem:PinelisBound} illustrates how martingale (difference) sequences are used in this argument, resulting in stability as stated in Proposition \ref{prop:stab} and, ultimately, in pathwise uniform convergence rates, see Theorem \ref{thm:convergenceRate}. It yields convergence of the pathwise uniform discretisation error \eqref{eq:uniformerrorIntro} of the regularised problem as $k \to 0$. Note that this convergence is not uniform in the regularisation parameter $m \in \N$.

The main task consists in proving convergence of the regularisation error both for the mild solutions to \eqref{eq:regStEvolEqnIntro} and for the discretisations \eqref{eq:regDiscretisationIntro} as $m \to \infty$ uniformly in the number of time steps $N_k$. For the continuous regularisation error, this is achieved by a combination of a maximal inequality for stochastic convolutions, continuity of paths of the nonlinearities evaluated at the mild solution, and a classical continuous Gronwall argument.

An analogous straightforward estimate of the discrete regularisation error fails. Instead, the maximal inequality for discrete convolutions already used in the stability proof proves helpful. In addition, a clever splitting of the error is required so that it can be rewritten in terms of the continuous regularisation error, the discretisation error of the regularised problem, as well as the full error, i.e., the discretisation error of the original problem \eqref{eq:stEvolEqnIntro}. The  regularisation parameter $m \in \N$ can then be fixed large enough such that the first error becomes small, and we already showed uniform convergence of the second as $k \to 0$. In the end, we thus derive an estimate for the full error in terms of itself, and we apply a standard discrete Gronwall argument, finally resulting in the desired convergence statement.

\subsection{Overview}

\begin{itemize}
\item Section \ref{sec:prelim} contains the preliminaries for the rest of the paper.
\item In Section \ref{sec:2smoothRates} we extend previous results on pathwise uniform convergence rates from Hilbert to $2$-smooth Banach spaces. In particular, we present a novel stability proof.
\item In Section \ref{sec:convergence} we state and prove the main result on pathwise uniform convergence of time discretisation schemes in the case of irregular Lipschitz nonlinearity and noise, which leads to Theorem \ref{thm:convergenceIntro}.
\item Section \ref{sec:example} illustrates the results for the nonlinear stochastic Schrödinger equation.
\end{itemize}

\subsubsection*{Acknowledgements}

The first author wishes to thank the DAAD for the financial support to visit TU Delft for one semester in 2022 and the colleagues in Delft for their hospitality.
The authors also thank Jan van Neerven, Christian Seifert for helpful discussion and comments. Furthermore, the authors would like to thank the referee for careful reading and pointing out the estimate \eqref{eq:easyEstPointwiseStrong}.

\section{Preliminaries}
\label{sec:prelim}

Throughout the paper, we consider the final time $T>0$ to be fixed and denote the Borel $\sigma$-algebra of a Banach space $X$ by $\calB(X)$. We use the notation $f(x) \lesssim g(x)$ to denote that there is a constant $C \ge 0$ such that for all $x$ in the respective set, $f(x) \le C g(x)$. Furthermore, we fix a probability space $(\Omega, \filtrF, \PP)$ and a filtration $(\filtrF_t)_{t \in [0,T]}$ on this probability space. Unless otherwise stated, all random variables and stochastic processes considered are defined on $(\Omega, \filtrF, \PP)$. The progressive $\sigma$-algebra on $(\Omega, \filtrF, \PP)$ is denoted by $\calP$.

\subsection{$2$-smooth Banach spaces}
\label{subsec:2smoothBanach}

In this paper, we will work in $2$-smooth Banach spaces, a generalisation of Hilbert spaces, which is characterised by a parallelogram inequality instead of a parallelogram identity as is the case for Hilbert spaces.

\begin{definition}
	Let $D \ge 1$. A {\em $(2,D)$-smooth Banach space} is a Banach space $X$ for which for all $x,y \in X$,
	\[ \|x+y\|^2+\|x-y\|^2 \le 2 \|x\|^2+2D^2 \|y\|^2.\]
	We call a Banach space {\em $2$-smooth} if it is $(2,D)$-smooth for some $D \ge 1$.
\end{definition}
In the realm of stochastic analysis, this class of spaces plays an important role. As a consequence of the parallelogram identity, it includes all Hilbert spaces with $D=1$. Furthermore, the spaces $L^p(\mu)$ are contained in this class for $2 \le p <\infty$ with $D=\sqrt{p-1}$ \cite[Proposition~2.1]{Pin}. The following simple fact will be used throughout the paper: If $X$ is $(2,D)$-smooth and $A$ is a closed linear operator, then $D(A)$ equipped with the graph norm $(\|x\|^2+\|Ax\|^2)^{1/2}$ is again $(2,D)$-smooth.

\subsection{$\gamma$-radonifying operators}
\label{subsec:radonifying}

To give sense to stochastic integrals in Banach spaces $X$ that are not Hilbert, the space of $\gamma$-radonifying operators $\gamma(H,X)$ is required, where $H$ denotes a Hilbert space. It is obtained as the closure of a subset of the space of $\gamma$-summing operators. In this subsection only, let $(\gamma_n)_{n \in \N}$ denote a Gaussian sequence, i.e., a sequence of standard Gaussian random variables.

\begin{definition}
    We call a linear operator $R\from H \to X$ {\em $\gamma$-summing} if
    \begin{equation}
    \label{eq:defgammaSumming}
        \|R\|_{\gamma^\infty(H,X)} \ce \sup \Big(\E\Big\|\sum_{n=1}^N \gamma_nRh_n\Big\|_X^2\Big)^{1/2} <\infty,
    \end{equation}
    with the supremum being taken over all finite orthonormal systems $\{h_1,\ldots,h_N\}$ in $H$ with $N \in \N$. The space of all operators for which \eqref{eq:defgammaSumming} holds is denoted by $\gamma^\infty(H,X)$. 
\end{definition}
The thus obtained $(\gamma^\infty(H,X),\|\cdot\|_{\gamma^\infty(H,X)})$ is a normed space, which is contained in the space of bounded linear operators $\calL(H,X)$. The inclusion follows immediately from considering orthonormal systems $\{h\}$ consisting of a single element. Furthermore, $\gamma^\infty(H,X)$ is a Banach space. 

The space of $\gamma$-radonifying operators from $H$ to $X$ is now obtained as the closure of finite rank operators in the space of $\gamma$-summing operators.

\begin{definition}
    Let $N \in \N$, $h_n \in H$, and $x_n \in X$ for $1 \le n \le N$ and define the operator $h_n \otimes x_n \in \calL(H,X)$ by $h \mapsto (h_n \otimes x_n)h \ce \langle h_n,h\rangle_H x_n$. Operators of the form $R=\sum_{n=1}^N h_n \otimes x_n$ are called {\em finite rank operators}, and the space of all such operators is denoted by $\mathrm{F}\mathrm{R}(H,X)$. We define the space $\gamma(H,X)$ of all {\em $\gamma$-radonifying operators} as the closure of $\mathrm{F}\mathrm{R}(H,X)$ in $\gamma^\infty(H,X)$. 
\end{definition}

Trivially, $\mathrm{F}\mathrm{R}(H,X) \seq \gamma^\infty(H,X)$. As a closed subspace of $\gamma^\infty(H,X)$, $\gamma(H,X)$ is a Banach space with the norm $\|\cdot\|_{\gHX} \ce \|\cdot\|_{\gamma^\infty(H,X)}$. For a finite rank operator $R=\sum_{n=1}^N h_n \otimes x_n \in \gHX$ with orthonormal $\{h_1,\ldots,h_N\} \seq H$, the norm \eqref{eq:defgammaSumming} simplifies to
\[ \|R\|_{\gHX}^2 = \E\bigg\|\sum_{n=1}^N \gamma_n x_n\Big\|^2.\]
In case $X$ is a Hilbert space, the norm of $R$ further simplifies to $\sum_{n=1}^N \|x_n\|^2$. Hence, by taking the completion, we see that $\gHX$ coincides with the space $\calL_2(H,X)$ of Hilbert-Schmidt operators for Hilbert spaces $X$ and $H$. An example for a $\gamma$-radonifying operator in a non-Hilbert space is given by the indefinite integration operator $I_T\from L^2(0,T) \to C[0,T]$ defined by $f \mapsto [t \mapsto \int_0^t f(s)\ds]$ for $t \in [0,T]$.

A property of $\gHX$ frequently used in the following is the (left) ideal property.

\begin{proposition}[\cite{AnalysisBanachSpacesII}, Theorem 9.1.10]
\label{prop:leftIdeal}
    Let $R \in \gamma^\infty(H,X)$. Let $\tilde{H}$ be another Hilbert space and $\tilde{X}$ another Banach space. Then for all $L_1 \in \calL(\tilde{H},H)$ and $L_2\in \calL(X, \tilde{X})$, we have $L_2RL_1 \in \gamma^\infty(\tilde{H},\tilde{X})$ and for all $1 \le p <\infty$,
    \[ \|L_2RL_1\|_{\gamma^\infty(\tilde{H},\tilde{X})} \le \|L_2\|_{\calL(X, \tilde{X})} \|R\|_{\gamma^\infty(H,X)} \|L_1\|_{\calL(\tilde{H},H)}.\]
    If, moreover, $R \in \gHX$, then $L_2RL_1 \in \gamma(\tilde{H},\tilde{X})$.
\end{proposition}

For details on the aforementioned and further properties of $\gamma$-radoniyfing operators, we refer the reader to \cite[Section~9.1]{AnalysisBanachSpacesII}.

\subsection{Stochastic integration}

Let $H$ denote a separable Hilbert and $X$ a $2$-smooth Banach space. For a $\gamma$-radonifying operator $R\in \gHX$ and a sequence $\gamma = (\gamma_n)_{n\in \N}$ of centered i.i.d.\ normally distributed random variables, we write 
\begin{equation}\label{eq:convradonW}
R \gamma \ce \sum_{n\in\N} \gamma_n R h_n
\end{equation}
for an orthonormal basis $\{h_n\}_{n \in \N}$ of $H$. The convergence is in $L^p(\Omega;X)$ for $p<\infty$ and almost surely, independently of the choice of orthonormal basis (see \cite[Corollary 6.4.12]{AnalysisBanachSpacesII}). While $R\gamma$ depends on the choice of the orthonormal basis, its distribution does not.
The stochastic integrals appearing, for instance, in the mild solution formula \eqref{eq:stEvolEqnIntro} are stochastic integrals of operator-valued integrands. Hence, the integrator is an $H$-cylindrical Brownian motion rather than a (standard) Brownian motion. 

An \emph{$H$-cylindrical Brownian motion} is a mapping $W_H\from L^2(0,T;H) \to L^2(\Omega)$ such that
\begin{enumerate}[label=(\roman*)]
    \item $W_H h$ is Gaussian for all $h \in L^2(0,T;H)$,
    \item $\E(W_H h_1 \cdot W_H h_2) = \langle h_1,h_2 \rangle_{L^2(0,T;H)}$ for all $h_1, h_2 \in L^2(0,T;H)$,
\end{enumerate}
where we include a complex conjugate on $W_H h_2$ in case we want to use a complex $H$-cylindrical Brownian motion.
As a shorthand notation, we let $W_H(t)h \ce W_H(\1_{(0,t)}\otimes h)$ for $h \in H$ and $t \in [0,T]$. An $H$-cylindrical Brownian motion can be regarded as the infinite-dimensional analogue of a Brownian motion, in the sense that $(W_H(t)h)_{t \in [0,T]}$ is a (standard) Brownian motion for each fixed $h \in H$ (of norm $\|h\|_H=1$). Real-valued Brownian motions are recovered in the case $H=\R$.

A notion closely related to $H$-cylindrical Brownian motions are so-called $Q$-Wiener processes. An $H$-valued stochastic process $(W(t))_{t \geq 0}$ is called a \emph{$Q$-Wiener process} for an operator $Q \in \calL(H)$ if $W(0)=0$, $W$ has continuous trajectories and independent increments, and $W(t)-W(s)$ is normally distributed with parameters $0$ and $(t-s)Q$ for $t\geq s \geq 0$. One can show that $W$ is a $Q$-Wiener process if and only if there exists an $H$-cylindrical Brownian motion $W_H$ such that $Q^{1/2}W_H\ce\sum_{n\geq 1} Q^{1/2} h_n W_H(t) h_n = W(t)$ for an orthonormal basis $(h_n)_{n \ge 1}$ of $H$ (cf. \eqref{eq:convradonW} and note that the series is independent of the choice of $(h_n)_{n \in \N}$). 

Stochastic integrals with respect to $H$-cylindrical Brownian motions or $Q$-Wiener processes can be defined in the sense of Itô integrals. Further properties of $H$-cylindrical Brownian motions, $Q$-Wiener processes, and details on the Itô integral in Hilbert spaces can be found in \cite{DaPratoZabczyk14}. An overview of stochastic integration in Banach spaces is contained in \cite{vanNeervenVeraarWeisSurvey}.

To estimate Itô integrals w.r.t. such $H$-cylindrical Brownian motions, the Burkholder-Davis-Gundy inequalities are particularly helpful. In the case $X$ is a Hilbert space, they imply that
\begin{equation}\label{eq:BDG}
\bigg(\E \sup_{t \in [0,T]} \bigg\| \int_0^t g_s \dWHs \bigg\|_X^p\bigg)^{1/p} \leq B_p \|g\|_{L^p(\Omega;L^2(0,T;\gHX))}
\end{equation}
for some constant $B_p>0$ for all $g \in L^p(\Omega;L^2(0,T;\gHX))$. In general $(2,D)$-smooth Banach spaces, \eqref{eq:BDG} holds with $B_p=10D\sqrt{p}$. This follows as a special case of the following maximal inequality for stochastic convolutions from \cite{JanMarkMainPaper} based on earlier works on the contractive case in Hilbert spaces \cite{HausSei}. 
We recall that a $C_0$-semigroup $(S(t))_{t \ge 0}$ is said to be \textit{quasi-contractive} if for some constant $\lambda\ge 0$, $\|S(t)\| \le e^{\lambda t}$ for all $t \ge 0$. 

\begin{theorem}[\cite{JanMarkMainPaper}, Theorem 4.1]
\label{thm:maxIneqQuasiContractive}
	Let $(S(t))_{t \ge 0}$ be a quasi-contractive $C_0$-semigroup with constant $\lambda>0$ on a $(2,D)$-smooth Banach space $X$. 

 \noindent
    Then for every $g \in L_\calP^0(\Omega;L^2(0,T;\gHX))$ the process $(\int_0^t S(t-s)g(s)\dWHs)_{t \in [0,T]}$ has a continuous modification, which satisfies, for all $0<p<\infty$,
    \begin{equation*}
        \E \sup_{t \in [0,T]} \bigg\| \int_0^t S(t-s)g(s) \dWHs \bigg\|^p \le C_{p,D}^p \|g\|_{L^p(\Omega;L^2(0,T;\gHX))}^p,
    \end{equation*}
    with a constant $C_{p,D}$ depending only on $p$ and $D$. For $2 \le p <\infty$ the inequality holds with $C_{p,D}=10\ee^{\lambda T}D \sqrt{p}$.
\end{theorem}
Recall that for $p<1$, the expression on the right is only a seminorm of $g$.
Considering $S$ as the trivial semigroup, we recover continuity of Itô's isomorphism. In the case $p=2$, $H=\R$, and $X$ Hilbert, it is even an isometry known as \emph{Itô's isometry}.

\subsection{A version of the Rosenthal–Burkholder inequality}

On the fixed probability space $(\Om,\F,\P)$, we consider a finite filtration $(\F_j)_{j= 0}^\ell$, $\ell\in \N$, and denote by $\E_{\F_j}\ce\E(\cdot\mid\F_j)$ the
conditional expectation with respect to $\F_{j}$.
For an $X$-valued martingale $ (M_j)_{j= 0}^\ell$ with respect to $(\F_j)_{j= 0}^\ell$, we denote by
$(d_j)_{j=1}^\ell$ its difference sequence defined by
$d_j \ce M_j - M_{j-1}$.
Furthermore, let the non-negative random variables $M_j\ss$ (for $0\le j\le \ell$) and $d_j\ss$ and $s_j(M)$ (for $1\le j\le \ell)$ be given by
\begin{align*}
 M_j\ss  \ce\max_{0\le i\le j} \n M_i\n , \quad d_j\ss \ce\max_{1\le i\le j} \n d_i\n,  \quad  s_j(M)\ce\Bigl(\sum_{i=1}^j \E_{\F_{i-1}} \n d_{i}\n^2\Bigr)^{1/2},
\end{align*}
and set $M\ss \ce M\ss_\ell$, $d\ss \ce d\ss_\ell$, and $s(M) \ce s_\ell(M)$.

We call a mapping $V\from \Om\to \calL(X)$ such that $\om\mapsto V(\om)x$ is strongly measurable for all $x\in X$ a {\em random operator} on $X$ and a {\em random contraction} on $X$ if, additionally, its range consists of contractions. A sequence of random operators $(V_j)_{j \in \N}$ on $X$ is said to be {\em strongly predictable} in case each $V_jx$ is strongly $\filtrF_{j-1}$-measurable for all $x \in X$.

An adapted $X$-valued sequence $(\xi_j)_{j=1}^\ell$ is called {\em conditionally symmetric given $(\F_j)_{j=0}^\ell$}
if for all $1\le j\le \ell$ the random variables
 $\xi_j$ and $-\xi_j$ are conditionally equi-distributed given $\F_{j-1}$, i.e., for all Borel sets $B\in \calB(X)$ it holds that 
 $$ \E_{\F_{j-1}} \one_{\{\xi_j \in B\}} = \E_{\F_{j-1}}\one_{\{-\xi_j \in B\}}.$$

Recently, in \cite[Theorem 3.1]{JanMarkMainPaper} an extended version of Pinelis's version of the Rosenthal--Burkholder inequality (see \cite{Pin}) was proven. An alternative approach based on Bellman function techniques was found in \cite{Zo-Kr}. 

\begin{theorem}[\cite{JanMarkMainPaper}, Theorem 3.1]
\label{thm:Pinelis}
    Let $X$ be a $(2,D)$-smooth Banach space. Suppose that $(\widetilde{M}_j)_{j=0}^\ell$ is an adapted sequence of $X$-valued random variables, $(M_j)_{j=0}^\ell$ is an $X$-valued martingale with difference sequence $(d_j)_{j=1}^\ell$, $(V_j)_{j=1}^\ell$ is a sequence of random contractions on $X$ that is strongly predictable, and assume that we have $\widetilde{M}_0=M_0=0$ and
    \begin{equation*}
        \widetilde{M}_j=V_j\widetilde{M}_{j-1}+d_j,\quad j=1,\ldots,l.
    \end{equation*}
    Then for all $2 \le p <\infty$ we have
    \begin{equation*}
        \|(\widetilde{M})^\ast\|_p \le 30p\|d^\ast\|_p + 40D\sqrt{p} \|s(M)\|_p.
    \end{equation*}
    If, moreover, $(M_j)_{j=0}^\ell$ has conditionally symmetric increments, then
    \begin{equation*}
        \|(\widetilde{M})^\ast\|_p \le 5p\|d^\ast\|_p + 10D\sqrt{p} \|s(M)\|_p.
    \end{equation*}
\end{theorem}

\subsection{Approximation of semigroups}

A fundamental part of the approximation of solutions to a stochastic evolution equation entails the temporal approximation of a semigroup by a scheme. The approximation behaviour is quantified as follows.

\begin{definition}
\label{def:orderScheme}
Let $X$ be a Banach space. An \emph{$\calL(X)$-valued scheme} is a function $R\from  [0,\infty) \to \calL(X)$. We denote $R_k \ce R(k)$ for $k \ge 0$. Let $Y$ be a Banach space that is continuously and densely embedded into $X$. If $A$ generates a $C_0$-semigroup $(S(t))_{t \ge 0}$ on $X$, an $\calL(X)$-valued scheme $R$ is said to \emph{approximate $S$ to order $\alpha>0$ on $Y$} or, equivalently, \emph{$R$ converges of order $\alpha$ on $Y$} if for all $T>0$ there is a constant $C_\alpha \ge 0$ such that
\begin{equation*}
    \|(S(jk)-R_k^j)u\|_X \le C_\alpha k^\alpha\|u\|_Y
\end{equation*}
for all $u \in Y$, $k>0$, and $j \in \N$ such that $jk \in [0,T]$.
An $\calL(X)$-valued scheme $R$ is said to be \emph{contractive} if $\|R_k\|_{\calL(X)} \le 1$ for all $k \ge 0$.
\end{definition}
Henceforth, the index for norms in the space $X$ will be omitted. In the linear deterministic case, the following schemes approximate $S$ to different orders:

\begin{itemize}
\item splitting scheme (S): $R_k = S(k)$, any order $\alpha >0$ on $X$;
\item implicit Euler (IE): $R_k = (1-kA)^{-1}$, order $\alpha \in (0,1]$ on $D((-A)^{2\alpha})$;
\item Crank-Nicolson (CN): $R_k = (2+kA)(2-kA)^{-1}$, order $\alpha \in (0,2]$ on $D((-A)^{3\alpha/2})$ provided that $R$ is contractive.
\end{itemize}

Note that contractivity of $(S(t))_{t\ge 0}$ implies sectoriality of $-A$ and thus the fractional powers $(-A)^\beta$ exist for $\beta>0$. As many commonly used schemes, (IE) and (CN) can be written as $R_k=r(-kA)$ for some function $r\from  \C_+ \to \C$, where $r(-kA)$ is defined via the $H^\infty$-calculus of $-A$. 

Common choices for the space $Y$ in Definition \ref{def:orderScheme} are domains of fractional powers of the negative of the generator $A$ of the semigroup. An important property of these spaces is their embedding into the real interpolation spaces with parameter $\infty$. That is, for $\alpha>0$
\begin{equation}
\label{eq:DAalphaembeds}
    D(A^{\alpha}) \hookrightarrow D_A(\alpha, \infty),
\end{equation}
where $D_A(\alpha,\infty)$ denotes the real interpolation space $(X,D(A))_{\alpha,\infty}$. See \cite{Lun,Tr1} for details on real interpolation spaces.

\subsection{Gronwall type lemmas}

We need the following variants of the classical Gronwall inequality from \cite[Lemmas 2.5 and 2.6]{KliobaVeraar23} in the continuous and the discrete version based on \cite[Proposition 5]{holteGronwall}.
\begin{lemma}\label{lem:gronwallvar}
Let $\phi\from [0,T]\to [0,\infty)$ be a continuous function and let $\alpha,\beta\ge 0$ be constants. If
\[\phi(t) \leq \alpha+\beta\Big(\int_0^t \phi(s)^2 ds\Big)^{1/2}, ~\text{ for } t\in [0,T],\]
then
\[\phi(t) \leq \alpha (1+\beta^2 t)^{1/2}\exp\Big(\frac12+\frac12 \beta^2 t\Big) , ~\text{ for } t\in [0,T].\]
\end{lemma}

\begin{lemma}
\label{lem:KruseGronwall}
    Let $(\varphi_j)_{j \ge 0}$ be a non-negative sequence and $\alpha,\beta \ge 0$. If
    \begin{equation*}
        \varphi_j \le \alpha + \beta\Big(\sum_{i=0}^{j-1} \varphi_i^2\Big)^{1/2}~\text{ for } j \ge 0,
    \end{equation*}
    then
    \begin{equation*}
        \varphi_j \le \alpha (1+\beta^2 j)^{1/2} \exp\Big(\frac12+\frac12\beta^2 j\Big)~\text{ for }j \ge 0.
    \end{equation*}
\end{lemma}

\section{Convergence rates on $2$-smooth Banach spaces}
\label{sec:2smoothRates}

We consider the stochastic evolution equation with multiplicative noise
\begin{align}
\label{eq:StEvolEqnFG_WP}
    \Bigg\{\begin{split} \rmd U &=(A U + F(t,U))\,\rmd t  + G(t,U) \,\rmd W_H~~~\text{ on } [0,T],\\ U(0) &= u_0 \in L_{\filtrF_0}^p(\Omega;X)
    \end{split}
\end{align}
for $2 \le p <\infty$, $T>0$, and $A$ generating a $C_0$-semigroup $(S(t))_{t \ge 0}$ of contractions on a $2$-smooth Banach space $X$. In this section, we extend the pathwise uniform convergence rates obtained in \cite{KliobaVeraar23} for contractive time discretisation schemes on Hilbert spaces to $2$-smooth Banach spaces. In particular, this includes a generalisation of the well-posedness and stability results used for the main error estimate in Theorem \ref{thm:convergenceRate}.

The following assumption ensures well-posedness of the stochastic evolution equation \eqref{eq:StEvolEqnFG_WP}.
\begin{assumption} 
\label{ass:WP}
    Let $X$ be a $(2,D)$-smooth Banach space for some $D \ge 1$ and let $p \in [2,\infty)$.
    Let $F\from \Omega \times [0,T] \times X \to X$, $F(\omega, t,x) = \tilde{F}(\omega,t,x) + f(\omega,t)$ and $G\from \Omega \times [0,T] \times X \to \gHX$, $G(\omega, t,x) = \tilde{G}(\omega,t,x) + g(\omega,t)$ be strongly $\calP\otimes \calB(X)$-measurable and such that $\tilde{F}(\cdot,\cdot,0) = 0$ as well as $\tilde{G}(\cdot,\cdot,0) = 0$. Suppose that
    \begin{enumerate}[label=(\alph*)]
        \item\label{assItem:WP_Lipschitz} \emph{(global Lipschitz continuity on $X$)} there exist constants $\CF, \CG \ge 0$ such that for all $\omega \in \Omega, t \in [0,T]$ and $x,y\in X$, it holds that
        \begin{align*}
            \|\tilde{F}(\omega,t,x)-\tilde{F}(\omega,t,y)\| &\le \CF\|x-y\|,\\ \|\tilde{G}(\omega,t,x)-\tilde{G}(\omega,t,y)\|_\gHX &\le \CG\|x-y\|,
        \end{align*}
        \item\label{item:FG_WP_integrability} \emph{(integrability)} $f \in L^p_\calP(\Omega; L^1(0,T;X))$ and $g \in L^p_\calP(\Omega; L^2(0,T;\gHX))$.
    \end{enumerate}
\end{assumption}
Note that Assumption \ref{ass:WP} implies linear growth of $F$ and $G$, i.e., for all $\omega \in \Omega$, $t \in [0,T]$, and $x \in X$,
\begin{equation*}
    \|\tilde{F}(\omega,t,x)\| \le \CF \|x\|\quad\text{ and }\quad \|\tilde{G}(\omega,t,x)\|_\gHX \le \CG \|x\|.
\end{equation*}

Well-posedness shall be understood in the sense of existence and uniqueness of mild solutions to \eqref{eq:StEvolEqnFG_WP}.
\begin{definition}
    A $U\in L^0_{\calP}(\Omega;C([0,T];X))$ is called a \emph{mild solution} to \eqref{eq:StEvolEqnFG_WP} if a.s.\ for all $t \in [0,T]$
    \begin{equation*}
        U(t) = S(t)u_0 + \int_0^t S(t-s) F(s,U(s)) \,\rmd s + \int_0^t S(t-s) G(s,U(s)) \,\rmd W_H(s).
    \end{equation*}
\end{definition}

As a shorthand notation, we write 
\begin{equation*}
    \|f\|_{p,q,Z} \ce \|f\|_{L^p(\Omega;L^q(0,T;Z))},\quad \nn g \nn_{p,q,Z} \ce \|g\|_{L^p(\Omega;L^q(0,T;\gamma(H,Z)))}
\end{equation*}
for $p \in [2,\infty)$, $q \in [1,\infty]$, and $Z \seq X$.
The following well-posedness result is an extension of \cite[Thm.~4.3]{KliobaVeraar23} to $2$-smooth Banach spaces. 
\begin{theorem}
\label{thm:wellposed}
    Suppose that Assumption \ref{ass:WP} holds for some $p\in [2,\infty)$. Let $A$ be the generator of a $C_0$-contraction semigroup $(S(t))_{t \ge 0}$ on $X$ and let $u_0 \in L_{\calF_0}^p(\Omega;X)$.
   Then \eqref{eq:StEvolEqnFG_WP} has a unique mild solution $U \in L^p(\Omega;C([0,T];X))$. Moreover, there is a constant $C\ge 0$ depending only on $p$, $D$, $T$, $\CF$ and $\CG$ such that
    \begin{align*}
        \|U\|_{L^p(\Omega;C([0,T];X))} \le C\big(1+\|u_0\|_{L^p(\Omega;X)} + \|f\|_{p,1,X}+ \nn g\nn_{p,2,X} \big).
    \end{align*}
\end{theorem}
\begin{proof}
    The statement follows from an application of the contraction mapping theorem to the fixed point functional
    \begin{equation*}
        \Gamma v(t) \ce S(t)u_0 + \int_0^t S(t-s)F(s,v(s))\,\rmd s + \int_0^t S(t-s) G(s,v(s))\,\rmd W_H(s).
    \end{equation*}
    in the adapted subspace of $L^p(\Omega;C([0,\delta];X))$, as shown in \cite[Thm.~4.3]{KliobaVeraar23} for Hilbert spaces $X$. In order to apply the methods used in \cite{KliobaVeraar23}, we replace the space of Hilbert-Schmidt operators $\calL_2(H,X)$ by the space of $\gamma$-radonifying operators $\gamma(H,X)$ and applying the maximal inequality from Theorem \ref{thm:maxIneqQuasiContractive} instead of the Burkholder-Davis-Gundy inequality.
\end{proof}

Under a linear growth assumption on $F$ and $G$ on a $2$-smooth Banach space $Y$ embedding into $X$, the problem \eqref{eq:StEvolEqnFG_WP} is also well-posed on $Y$, as the following straightforward extension of \cite[Thm.~4.4]{KliobaVeraar23} illustrates.

\begin{theorem} 
\label{thm:wellposedY}
    Let $X$ and $Y$ be $(2,D)$-smooth Banach spaces, $Y \hra X$ and assume $A$ generates a $C_0$-contraction semigroup $(S(t))_{t \ge 0}$ on both $X$ and $Y$. Let $p \in [2,\infty)$, $u_0\in L^p_{\calF_0}(\Omega;Y)$ and let Assumption \ref{ass:WP} hold. Additionally, suppose that $f\in L^p_\calP(\Omega;L^1(0,T;Y))$, $g\in L_\calP^p(\Omega;L^2(0,T;\gHY))$, $F\from \Omega \times [0,T] \times Y \to Y$, $G\from  \Omega \times [0,T] \times Y \to \gHY$, and there are $\LF,\LG \ge 0$ such that for all $\omega \in\Omega$, $t \in [0,T]$ and $x \in Y$,
    \begin{equation*}
        \|\tilde{F}(\omega,t,x)\|_Y \le \LF(1+\|x\|_Y),~
        \|\tilde{G}(\omega,t,x)\|_\gHY \le \LG(1+\|x\|_Y).
    \end{equation*}
    Under these conditions, the unique mild solution $U\in L^p(\Omega;C([0,T];X))$ to \eqref{eq:StEvolEqnFG_WP} is in $L^p(\Omega;C([0,T];Y))$ and there is a constant $C\ge 0$ depending only on $p$, $D$, $T$, $\LF$, $\LG$ and $Y$ such that
    \begin{align*}
        \|U\|_{L^p(\Omega;C([0,T];Y))}\leq C\big(1+\|u_0\|_{L^p(\Omega;Y)}+ \|f\|_{p,1,Y}+\nn g\nn_{p,2,Y}\big).
    \end{align*}
\end{theorem}

Having established well-posedness, we now turn to stability. Let $R_k\from  X \to X$ be a time discretisation scheme with time step $k>0$ on a uniform grid $\{t_j=jk:~j=0,\ldots, N_k\}\subseteq [0,T]$ with final time $T=t_{N_k}>0$ and $N_k=\frac{T}{k}\in \N$ being the number of time steps. We consider the temporal approximations of the mild solution to \eqref{eq:StEvolEqnFG_WP}
given by $U^0 \ce u_0$ and
\begin{align}
\label{eq:defUjStab}
    U^j &\ce R_k U^{j-1} + k R_k F(t_{j-1},U^{j-1})+ R_k G(t_{j-1},U^{j-1})\Delta W_j
\end{align}
with Wiener increments $\Delta W_j \ce W_H(t_j)-W_H(t_{j-1})$ (see \eqref{eq:convradonW}). The above definition of $U^j$ can be reformulated as the discrete variation-of-constants formula
\begin{equation}
\label{eq:VoCUStab}
    U^j = R_k^ju_0+k\sum_{i=0}^{j-1}R_k^{j-i}F(t_i,U^i)+\sum_{i=0}^{j-1} R_k^{j-i} G(t_i,U^i)\Delta W_{i+1}
\end{equation}
for $j=0,\ldots,N_k$. 

The following stability result is a generalisation of \cite[Prop.~5.1]{KliobaVeraar23} to $2$-smooth Banach spaces. This requires replacing the dilation argument in the original proof by a martingale one based on Theorem \ref{thm:Pinelis}, which is the subject of the following lemma.

\begin{lemma}
\label{lem:PinelisBound}
    Let $X$ be a $(2,D)$-smooth Banach space, $N \in \N$ with $N \le N_k$, and $Q\from \Omega \times [0,T] \to \gHX$ be such that $Q_i \ce Q(\cdot,t_i)\in L^p(\Omega;\gHX)$ is $\calF_{t_i}$-measurable for $0 \le i \le N-1$. Suppose that $(R_k)_{k>0}$ is contractive. Then there is a constant $B_{p,D}\ge 0$ depending on $p$ and $D$ such that
    \begin{align*}
        \bigg\| \max_{0 \le j \le N}\Big\|\sum_{i=0}^{j-1} R_k^{j-i}Q_i \Delta W_{i+1}\Big\|\bigg\|_p \le B_{p,D} \Big(k\sum_{i=0}^{N-1} \|Q_i\|_{L^p(\Omega;\gHX)}^2\Big)^{1/2}.
    \end{align*}
\end{lemma}
\begin{proof}
    The bound follows from an application of Theorem \ref{thm:Pinelis} as illustrated in Step $1$ and a simplification of the two terms emerging in Steps $2$ and $3$.

    \textit{Step 1.} Define $\widetilde{M}_j \ce \sum_{i=0}^{j-1} R_k^{j-i}Q_i \Delta W_{i+1}$ and $M_j \ce \sum_{i=1}^j R_kQ_{i-1}\Delta W_i$ for $0 \le j \le N$. Further, define $d_j \ce M_j-M_{j-1}$ and $V_j \ce R_k$ for all $1 \le j \le N$. Then $\widetilde{M}_0=M_0=0$ by construction and $(\tilde{M}_j)_{j=0}^N$ is adapted because $Q_i$ is $\filtrF_{t_i}$- and thus also $\filtrF_{t_{j-1}}$-measurable and $\Delta W_{i+1}$ is $\filtrF_{t_{i+1}}$- and thus also $\filtrF_{t_j}$-measurable. Furthermore, $(M_j)_{j=0}^{N}$ is an $X$-valued martingale with conditionally symmetric increments since it is adapted and for $0 \le \ell \le j \le N$
    \begin{align*}
        \E(M_j|\filtrF_{t_\ell}) = \sum_{i=1}^j R_kQ_{i-1} \E(\Delta W_i|\filtrF_{t_\ell})
        = \sum_{i=1}^\ell R_kQ_{i-1}\Delta W_i = M_\ell
    \end{align*}
    by independence of $\Delta W_i$ of $\filtrF_{t_\ell}$ for all $i \ge \ell+1$. Consequently, $(d_j)_{j=1}^N$ is a martingale difference sequence, and Theorem \ref{thm:Pinelis} is applicable. It yields the bound
    \begin{align}
    \label{eq:PinelisBoundUnsimplified}
        \bigg\| \max_{0 \le j \le N}\Big\|\sum_{i=0}^{j-1} R_k^{j-i}Q_i \Delta W_{i+1}\Big\|\bigg\|_p = \|(\widetilde{M})^\ast\|_p \le 5p\|d^\ast\|_p+10D\sqrt{p}\|s(M)\|_p,
    \end{align}
    where $d^\ast = \max_{1 \le j \le N} \|d_j\|$ and $s(M)^2= \sum_{i=0}^{N-1} \E(\|M_{i+1}-M_i\|^2|\mid\filtrF_{t_i})$.

    \textit{Step 2.} To simplify the first term, we first apply the triangle inequality and Doob's maximal inequality \cite[Thm.~3.2.2]{AnalysisBanachSpacesI} before rewriting the Wiener increments as stochastic integrals to apply Itô's isomorphism as in Theorem \ref{thm:maxIneqQuasiContractive}. Lastly, making use of Minkowski's inequality in $L^{p/2}(\Omega)$, contractivity of $R_k$ and the dominated convergence theorem in $L^p(\Omega)$, it follows that
    \begin{align}
    \label{eq:d*}
        \|d^*\|_p &= \Big\|\max_{1 \le j \le N} \|M_j-M_{j-1}\|\Big\|_p \le 2 \|M^*\|_p \le \frac{2p}{p-1} \|M_{N}\|_p\nonumber\\
        &= \frac{2p}{p-1} \bigg\| \sum_{i=0}^{N-1} \int_0^{t_N} \1_{(t_i,t_{i+1}]}(s) R_kQ_i\dWHs\bigg\|_{L^p(\Omega;X)}\nonumber\\
        &\le \frac{2pC_{p,D}}{p-1} \bigg\|\Big(\int_0^{t_N}\Big\| \sum_{i=0}^{N-1} \1_{(t_i,t_{i+1}]}(s) R_kQ_i \Big\|_\gHX^2 \ds\Big)^{1/2}\bigg\|_p\nonumber\\
        &= \frac{2pC_{p,D}}{p-1} \bigg\|\sum_{i=0}^{N-1}\int_{t_i}^{t_{i+1}}\| R_kQ_i\|_\gHX^2 \ds\bigg\|_{p/2}^{1/2}\nonumber\\
        &\le \frac{2pC_{p,D}}{p-1} \Big(k \sum_{i=0}^{N-1}\big\|\| Q_i\|_\gHX^2 \big\|_{p/2}\Big)^{1/2}
        = \frac{2pC_{p,D}}{p-1} \Big(k \sum_{i=0}^{N-1}\| Q_i\|_{L^p(\Omega;\gHX)}^2 \Big)^{1/2} .
    \end{align}

    \textit{Step 3.} Using that the Wiener increments $\Delta W_{i+1}$ are independent of $\calF_{t_i}$ and have variance $t_{i+1}-t_i=k$, we can bound the remaining term $\|s(M)\|_p$ in \eqref{eq:PinelisBoundUnsimplified} by
    \begin{align}
    \label{eq:s(M)}
        \|s(M)\|_p &= \bigg\|\Big(\sum_{i=0}^{N-1} \E\big(\|R_kQ_i\Delta W_{i+1}\|^2\mid \calF_{t_i}\big) \Big)^{1/2}\bigg\|_p\nonumber\\
        &\le \bigg\|\Big(\sum_{i=0}^{N-1} \|R_kQ_i\|_\gHX^2 \E(|\Delta W_{i+1}|^2\mid \filtrF_{t_i})\Big)^{1/2}\bigg\|_p \nonumber\\
        &=  \bigg\| k\sum_{i=0}^{N-1} \|R_kQ_i\|_\gHX^2\bigg\|_{p/2}^{1/2} 
        \le \Big(k \sum_{i=0}^{N-1} \| Q_i\|_{L^p(\Omega;\gHX)}^2\Big)^{1/2}.
    \end{align}
    The statement of the theorem is obtained with $B_{p,D}=10p^2(p-1)^{-1}C_{p,D}+10D\sqrt{p}$ from inserting \eqref{eq:d*} and \eqref{eq:s(M)} in \eqref{eq:PinelisBoundUnsimplified}.
\end{proof}

\begin{proposition}
\label{prop:stab}
    Let $X$ be a $(2,D)$-smooth Banach space, $p \in [2, \infty)$ and $u_0 \in L^p(\Omega;X)$. Suppose that Assumption \ref{ass:WP} holds and that $(R_k)_{k>0}$ is contractive. Then the discrete solution \eqref{eq:defUjStab} obtained using $(R_k)_{k>0}$ is stable. More precisely, for all $T>0$ there exists a constant $C_T$ independent of $N_k$ such that
    \begin{equation*}
        \bigg\| \max_{0 \le j \le N_k} \|U^j\|\bigg\|_p \le C_T\big(\|f\|_{p,\infty,X}+\nn g\nn_{p,\infty,X}+\|u_0\|_{L^p(\Omega;X)} \big) < \infty.
    \end{equation*}
\end{proposition}
\begin{proof}
    Let $N \in \{0,\ldots,N_k\}$ and $\vp_N \ce \| \max_{0 \le j \le N} \|U^j\|\|_p$. Then the variation-of-constants formula \eqref{eq:VoCUStab} and contractivity of $R_k$ allow us to bound
    \begin{align}
    \label{eq:boundStabilityProof}
        \vp_N &\le \|u_0\|_{L^p(\Omega;X)} + k \sum_{i=0}^{N-1} \bigg\|\max_{0 \le j \le i}\|F(t_j,U^j)\|\bigg\|_p+ \bigg\|\max_{0 \le j \le N} \bigg\|\sum_{i=0}^{j-1} R_k^{j-i}G(t_i,U^i) \Delta W_{i+1}\bigg\|\bigg\|_p.
    \end{align}
    Invoking linear growth of $\tilde{F}$ and pathwise continuity of $f$ for the second term, we obtain the bound
    \begin{align}
    \label{eq:stabProof2ndTerm}
        k\sum_{i=0}^{N-1} \Big\|\max_{0 \le j \le i}\|F(t_j,U^j)\|\Big\|_p 
        &\le k\sum_{i=0}^{N-1}\Big(\CF\Big\|\max_{0\le j \le i}\|U^j\| \Big\|_p + \|f\|_{p,\infty,X}\Big)\nonumber\\
        &= t_N \|f\|_{p,\infty,X}  + \CF k \sum_{i=0}^{N-1} \varphi_i \le t_N \|f\|_{p,\infty,X}  + \CF \sqrt{t_N} \Big(k \sum_{i=0}^{N-1} \varphi_i^2 \Big)^{1/2},
    \end{align}
    where we have used the Cauchy-Schwarz inequality in the last step.\\
    To the last term in \eqref{eq:boundStabilityProof} we apply Lemma \ref{lem:PinelisBound} with $Q_i\ce G(t_i,U^i)$ for $0 \le i \le N-1$, which together with linear growth of $\tilde{G}$ yields
    \begin{align}
    \label{eq:stabProof3rdTerm}
        \bigg\|\max_{0 \le j \le N} \bigg\|\sum_{i=0}^{j-1}  R_k^{j-i}G(t_i,U^i) \Delta W_{i+1}\bigg\|\bigg\|_p &\le  B_{p,D} \Big( k \sum_{i=0}^{N-1} \|G(t_i,U^i)\|_{L^p(\Omega;\gHX)}^2 \Big)^{1/2}\nonumber\\
        &\le B_{p,D} \Big( k \sum_{i=0}^{N-1} \big(\nn g\nn_{p,\infty,X}+\CG\|U^i\|_{L^p(\Omega;X)}\big)^2 \Big)^{1/2}\nonumber\\
        &\le \sqrt{2} B_{p,D} \sqrt{t_N} \nn g\nn_{p,\infty,X} +\sqrt{2} B_{p,D} \CG \Big( k \sum_{i=0}^{N-1} \vp_i^2 \Big)^{1/2}.
    \end{align}
   Inserting \eqref{eq:stabProof2ndTerm} and  \eqref{eq:stabProof3rdTerm} in \eqref{eq:boundStabilityProof} followed by an application of the discrete version of Gronwall's inequality from Lemma \ref{lem:KruseGronwall} results in 
    \begin{align*}
        \vp_N^2&\le \sqrt{C_{\ee}}\ee^{C_{\ee}/2}\big(\|u_0\|_{L^p(\Omega;X)} + t_N \|f\|_{p,\infty,X}+ \sqrt{2} B_{p,D} \sqrt{t_N} \nn g\nn_{p,\infty,X} \big)
    \end{align*}
    with $C_{\ee}\ce 1+\CF^2t_N+2B_{p,D}^2\CG^2t_N$, which implies the desired statement for $N=N_k$ noting that $t_{N_k}=T$.
\end{proof}

Under the assumption of additional regularity in the structure of $F$ and $G$ and smooth initial data $u_0$, pathwise uniform convergence rates are obtained. We would like to draw the reader's attention to the difference in notation to Section \ref{sec:convergence}: As no additional regularity in the structure of nonlinearity and noise is assumed in that section, the following error estimate from Theorem \ref{thm:convergenceRate} does not apply to $U$ but only to its regularised counterpart $\mU$. This is the subject of Corollary \ref{cor:regDiscrError}.

\begin{assumption}\label{ass:FG_structured}
    Let $X,Y$ be $2$-smooth Banach spaces such that $Y \hra X$ continuously, and let $p \in [2,\infty)$.
    Let $F\from \Omega \times [0,T] \times X \to X, F(\omega, t,x) = \tilde{F}(\omega,t,x) + f(\omega,t)$ and $G\from \Omega \times [0,T] \times X \to \gHX, G(\omega, t,x) = \tilde{G}(\omega,t,x) + g(\omega,t)$ be strongly $\calP\otimes \calB(X)$-measurable and such that $\tilde{F}(\cdot,\cdot,0) = 0$ as well as $\tilde{G}(\cdot,\cdot,0) = 0$. Suppose that
    \begin{enumerate}[label=(\alph*)]
        \item\label{item:FGstruct_Lipschitz} \emph{(global Lipschitz continuity on $X$)} there exist constants $\CF, \CG \ge 0$ such that for all $\omega \in \Omega, t \in [0,T]$, and $x,y\in X$, it holds that
        \begin{align*}
            \|\tilde{F}(\omega,t,x)-\tilde{F}(\omega,t,y)\| &\le \CF\|x-y\|,\\
            \|\tilde{G}(\omega,t,x)-\tilde{G}(\omega,t,y)\|_\gHX &\le \CG\|x-y\|,
        \end{align*}
        \item \label{item:FGstruct_Hölder} \emph{(Hölder continuity with values in $X$)} for some $\alpha \in (0,1]$,
        \begin{align*}
            C_{\alpha,F}\ce \sup_{\omega\in \Omega, x \in X} [F(\omega,\cdot,x)]_\alpha <\infty,~C_{\alpha,G}\ce \sup_{\omega\in \Omega, x \in X} [G(\omega,\cdot,x)]_\alpha <\infty,
        \end{align*}
        \item \label{item:FGstruct_Yinvariance} \emph{($Y$-invariance)} $F\from \Omega \times [0,T] \times Y \to Y$ and $G\from \Omega \times [0,T] \times Y \to \gHY$ are strongly $\calP\otimes \calB(Y)$-measurable, $f \in L^p_\calP(\Omega; C([0,T];Y))$, and $g \in L^p_\calP(\Omega; C([0,T];\gHY))$,
        \item \label{item:FGstruct_linearGrowth}  \emph{(linear growth on $Y$)} there exist constants $\LF, \LG \ge 0$ such that for all $\omega \in \Omega, t \in [0,T]$, and $x\in Y$, it holds that
        \begin{equation*}
            \|\tilde{F}(\omega,t,x)\|_Y \le \LF(1+\|x\|_Y),~ \|\tilde{G}(\omega,t,x)\|_\gHY \le \LG(1+\|x\|_Y).
        \end{equation*}
    \end{enumerate}
\end{assumption}

Along the lines of the proof of \cite[Thm.~6.3]{KliobaVeraar23} in the Hilbert space case, we obtain the following error estimate in $2$-smooth Banach spaces. 

\begin{theorem}
\label{thm:convergenceRate}
    Suppose that Assumption \ref{ass:FG_structured} holds for some $\alpha \in (0,1]$ and $p \in [2,\infty)$. Let $A$ be the generator of a $C_0$-contraction semigroup $(S(t))_{t \ge 0}$ on both $X$ and $Y$.
    Let $(R_{k})_{k>0}$ be a time discretisation scheme that is contractive on $X$ and $Y$. Assume $R$ approximates $S$ to order $\alpha$ on $Y$. Suppose that $Y \hra D_A(\alpha,\infty)$ continuously if $\alpha \in (0,1)$ or $Y \hra D(A)$ continuously if $\alpha=1$. Let $u_0 \in L_{\calF_0}^p(\Omega;Y)$. Denote by $U$ the mild solution of \eqref{eq:StEvolEqnFG_WP} and by $(U^j)_{j=0,\ldots,N_k}$ the temporal approximations as defined in \eqref{eq:defUjStab}. Then for $N_k \ge 8$ there are constants $C_1,C_2,C_3,C_4 \ge 0$ independent of $k$ such that
    \begin{equation*}
        \bigg\lVert\max_{0 \le j \le N_k}\|U(t_j)-U^j\|\bigg\rVert_p
        \le C_1 k+ C_2k^{1/2}
        +\Big(C_3 +C_4\log \Big(\frac{T}{k}\Big)\Big)k^{\alpha}.
    \end{equation*}
    In particular, the approximations $(U^j)_j$ converge at rate $\min\{\alpha,\frac{1}{2}\}$ up to a logarithmic correction factor as $k \to 0$.
\end{theorem}

As in the Hilbert space case, the logarithmic factor is not required if the splitting scheme is used, which is given by $R_k=S(k)$.

\begin{remark}
    Provided that the semigroup $(S(t))_{t \ge 0}$ has a dilation, the error bound from Theorem \ref{thm:convergenceRate} extends to the full time interval $[0,T]$. More precisely, for a piecewise constant extension $\tilde{U}\from [0,T] \to L^p(\Omega;X)$ of $(U^j)_{j=0,\ldots,N_k}$, there is a constant $C \ge 0$ such that
    \begin{equation}
    \label{eq:errorFullInterval}
        \bigg\lVert\sup_{t\in [0,T]}\|U(t)-\tilde{U}(t)\|\bigg\rVert_p
        \le Ck^{1/2}+C
        \Big(1 +\log \Big(\frac{T}{k}\Big)\Big)k^{\alpha}.
    \end{equation}
    under the assumptions of Theorem \ref{thm:convergenceRate} and an additional integrability assumption on $f$ and $g$.  If the time discretisation scheme used is the splitting scheme, i.e., $R_k=S(k)$ for $k>0$, the optimal rate $(1+\sqrt{\log(T/k)})k^{1/2}$ can be obtained in \eqref{eq:errorFullInterval}. It is optimal as it coincides with the modulus of continuity of the Brownian motion. The proof of \eqref{eq:errorFullInterval} carries over verbatim from the Hilbert space case, given the dilation of the semigroup. For further details, we refer to \cite[Section~6.3]{KliobaVeraar23}.
    
    There are two main cases known in which semigroups on non-Hilbert spaces have a dilation: positive semigroups on $L^p$-spaces for $2<p<\infty$ \cite{Fendler1997} and (analytic) semigroups whose generator admits an $H^\infty$-calculus of angle less than $\frac{\pi}{2}$ \cite{FroehlichWeis2006}.
\end{remark}

\section{Pathwise Uniform Convergence for Irregular Nonlinearities}
\label{sec:convergence}

Our aim is to prove pathwise uniform convergence of contractive time discretisation schemes for nonlinear stochastic evolution equations of the form 
\begin{equation}
\label{eq:StEvolEqnFG}
    \rmd U = (AU + F(t,U))\,\rmd t + G(t,U)\,\rmd W_H(t),~~U(0)=u_0 \in L^p(\Omega;X)
\end{equation}
with $t \in [0,T]$, $T>0$, on a $2$-smooth Banach space $X$ with norm $\|\cdot\|$, where $W_H$ is an $H$-cylindrical Brownian motion for some Hilbert space $H$ and $p \in [2,\infty)$. The operator $A$ is assumed to generate a contractive $C_0$-semigroup $(S(t))_{t\ge 0}$ on $X$. The main novelty of our paper lies in not assuming further regularity in the structure of the nonlinearity or the noise. That is, we merely assume that $F\from \Omega \times [0,T]\times X \to X$ and $G\from \Omega \times [0,T]\times X \to \gHX$ and impose no further conditions on the images $F(\Omega \times [0,T]\times Y)$ for some $Y \hra X$ or even on $F(\Omega \times [0,T]\times X)$ being proper, more regular subspaces of $X$. Moreover, we allow rough initial data $u_0 \in L^p(\Omega;X)$. 
Further assuming progressive measurability and global Lipschitz continuity of $F$ and $G$ as detailed in Assumption \ref{ass:FG_conv}, we have the existence of the unique mild solution to \eqref{eq:StEvolEqnFG} given by a fixed point of
\begin{equation}
\label{eq:mildSoldefU}
    U(t) = S(t)u_0 + \int_0^t S(t-s)F(s,U(s))\,\rmd s + \int_0^t S(t-s)G(s,U(s))\,\rmd W_H(s)
\end{equation}
for $t \in [0,T]$, cf. Theorem \ref{thm:wellposed}.

For time discretisation, we employ a contractive time discretisation scheme $R\from  [0,\infty) \to \calL(X)$ with time step $k>0$ on a uniform grid $\{t_j=jk:~j=0,\ldots, N_k\}\subseteq [0,T]$ with final time $T=t_{N_k}>0$ and $N_k=\frac{T}{k}\in \N$ being the number of time steps. As in the previous section, the discrete solution is given by $U^0 \ce u_0$ and
\begin{align}
\label{eq:defUj}
    U^j &\ce R_k U^{j-1} + k R_k F(t_{j-1},U^{j-1})+ R_k G(t_{j-1},U^{j-1})\Delta W_j
\end{align}
for $j=1,\ldots,N_k$ with Wiener increments $\Delta W_j \ce W_H(t_j)-W_H(t_{j-1})$.

We summarise the conditions imposed on $F$ and $G$.

\begin{assumption}
\label{ass:FG_conv}
    Let $X$ be a $(2,D)$-smooth Banach space for some $D \ge 1$ and let $p \in [2,\infty)$.
    Let $F\from \Omega \times [0,T] \times X \to X$ and $G\from \Omega \times [0,T] \times X \to \gHX$ be strongly $\calP\otimes \calB(X)$-measurable. Suppose that
    \begin{enumerate}[label=(\alph*)]
        \item\label{item:FGconv_Lipschitz} \emph{(global Lipschitz continuity)} there exist constants $\CF, \CG \ge 0$ such that for all $\omega \in \Omega, t \in [0,T]$, and $x,y\in X$, it holds that
        \begin{align*}
            \|F(\omega,t,x)-F(\omega,t,y)\| &\le \CF\|x-y\|,\\ \|G(\omega,t,x)-G(\omega,t,y)\|_\gHX &\le \CG\|x-y\|,
        \end{align*}
        \item \emph{(linear growth)} there exist constants $\LF, \LG \ge 0$ such that for all $\omega \in \Omega, t \in [0,T]$, and $x\in X$, it holds that
        \begin{align*}
            \|F(\omega,t,x)\| &\le \LF(1+\|x\|),\quad \|G(\omega,t,x)\|_\gHX \le \LG(1+\|x\|),
        \end{align*}
        \item\label{item:FGconv_Hölder} \emph{(Hölder continuity)} for some $\alpha \in (0,1]$, 
        \begin{align*}
            C_{\alpha,F}\ce \sup_{\omega\in \Omega, x \in X} [F(\omega,\cdot,x)]_\alpha <\infty,~C_{\alpha,G}\ce \sup_{\omega\in \Omega, x \in X} [G(\omega,\cdot,x)]_\alpha <\infty.
        \end{align*}
    \end{enumerate}
\end{assumption}

Note that Hölder continuity of $F$ as stated above implies pathwise uniform continuity. Condition \ref{item:FGconv_Hölder} can be weakened to
the existence of some $\alpha \in (0,1]$ such that
\begin{equation*}
    \sup_{x \in X} \sup_{0 \le s \le t \le T} \frac{F(\cdot,t,x)-F(\cdot,s,x)}{(t-s)^\alpha} \in L^p(\Omega)
\end{equation*}
 and likewise for $G$, i.e., pathwise Hölder continuity uniformly in $x \in X$ is sufficient together with existence of $p$-th moments of the Hölder seminorms. Assumption \ref{ass:FG_conv} implies Assumption \ref{ass:WP} with $\tilde{F} \ce F-F(\cdot,\cdot,0)$, $f\ce F(\cdot,\cdot,0)$, and likewise for $G$, whence \eqref{eq:StEvolEqnFG} has a unique mild solution. Compared to Assumption \ref{ass:WP}, only condition \ref{item:FGconv_Hölder} is added. Under these assumptions, we obtain the main result of this paper on pathwise uniform convergence of the temporal approximations.

\begin{theorem}
\label{thm:convergence}
    Suppose that Assumption \ref{ass:FG_conv} holds for some $p\in [2,\infty)$ and $\alpha\in (0,1]$. Further, suppose that $A$ generates a $C_0$-contraction semigroup on both $X$ and $D(A)$, and let $u_0 \in L_{\calF_0}^p(\Omega;X)$. Let $(R_k)_{k>0}$ be a time discretisation scheme that is contractive on both $X$ and $D(A)$ and that approximates $S$ to order $\alpha$ on $D(A)$. 
    Denote by $U$ the mild solution of \eqref{eq:StEvolEqnFG} and by $(U^j)_{j=0,\ldots,N_k}$ the temporal approximations as defined in \eqref{eq:defUj}. Define the piecewise constant extension $\tilde{U}\from [0,T] \to L^p(\Omega;X)$ by $\tilde{U}(t) \ce U^j$ for $t \in [t_j,t_{j+1})$, $0 \le j \le N_k-1$, and $\tilde{U}(T) \ce U^{N_k}$. Then
    \begin{align}
        \lim_{k \to 0} \bigg\| \sup_{t \in [0,T]} \|U(t)-\tilde{U}(t)\|\bigg\|_p =0.
    \end{align}
\end{theorem}

The main ingredient of the proof of this theorem consists of regularising the nonlinearity, the noise, and the initial values by 
\begin{equation}
\label{eq:defFnGnuon}
    \mF \ce mR(m,A)F,\quad \mG \ce mR(m,A)G,\quad \muo \ce mR(m,A)u_0
\end{equation}
for $m \in \N$. By construction, $\mF$ maps to $D(A)$, $\mG$ maps to $\gamma(H,D(A))$, and $\muo \in L^p(\Omega;D(A))$, giving the desired additional regularity in structure. Assumption \ref{ass:FG_conv} also implies existence and uniqueness of the mild solution $\mU$ of the regularised problem
\begin{equation}
\label{eq:StEvolEqnDefnU}
   \mU = (A\mU + \mF(\mU))\,\rmd t + \mG(\mU)\,\rmd W_H(t),~~\mU(0)=\muo \in X
\end{equation}
for $m \in \N$. It is given by a fixed point of
\begin{equation*}
    \mU(t) = S(t)\muo + \int_0^t S(t-s)\mF(\mU(s))\,\rmd s + \int_0^t S(t-s)\mG(\mU(s))\,\rmd W_H(s).
\end{equation*}

The following proposition lists useful properties of the regularised quantities. 
\begin{proposition}
\label{prop:prop}
    Let Assumption \ref{ass:FG_conv} hold and let $\mF,\mG,\muo$ be as defined in \eqref{eq:defFnGnuon} for $m \in \N$. Suppose that $A$ generates a $C_0$-contraction semigroup $(S(t))_{t \ge 0}$. Let $u_0 \in L_{\calF_0}^p(\Omega;X)$. Then the following statements hold. 
    \begin{enumerate}
        \item \emph{($D(A)$-invariance)} $\mF\from \Omega \times [0,T]\times D(A) \to D(A)$ and $\mG\from \Omega \times [0,T]\times D(A) \to \gamma(H,D(A))$ are strongly $\calP \otimes \calB(D(A))$-measurable and $\muo \in L_{\calF_0}^p(\Omega;D(A))$.
        \item \emph{(uniform Lipschitz continuity)} There are $\CF,\CG \ge 0$ such that for all $m \in \N$, $\omega \in \Omega, t \in [0,T]$, and $x,y\in X$, it holds that
        \begin{align*}
            \|\mF(\omega,t,x)-\mF(\omega,t,y)\| &\le \CF\|x-y\|,\\
            \|\mG(\omega,t,x)-\mG(\omega,t,y)\|_\gHX &\le \CG\|x-y\|.
        \end{align*}
        \item \emph{(linear growth on $D(A)$)} For all $m \in \N$, there are constants $\LFm, \LGm \ge 0$ such that for all $\omega \in \Omega, t \in [0,T]$, and $x\in Y$, it holds that
        \begin{align*}
            \|\mF(\omega,t,x)\|_{D(A)} &\le \LFm(1+\|x\|_{D(A)}),\\
            \|\mG(\omega,t,x)\|_{\gamma(H,D(A))} &\le \LGm(1+\|x\|_{D(A)}).
        \end{align*}
        \item \emph{(pointwise convergence)} As $m \to \infty$, $\mF$ and $\mG$ converge pointwise to $F$ and $G$, respectively. Moreover, $\muo \to u_0$ in $L^p(\Omega;X)$ as $m \to \infty$.
    \end{enumerate}
\end{proposition}
\begin{proof}
    \begin{enumerate}
        \item Continuity of $F(\omega,t,\cdot)\from  X \to X$ follows from Assumption \ref{ass:WP}\ref{assItem:WP_Lipschitz} for all $\omega \in \Omega, t \in [0,T]$, and thus also continuity as a mapping $F(\omega,t,\cdot)\from  D(A) \to X$. From the identity $AR(m,A)=mR(m,A)-I$ and continuity of the resolvent on $X$, we obtain continuity of $R(m,A)\from X \to D(A)$. Consequently, $\mF(\omega,t,\cdot)\from  D(A) \to D(A)$ is continuous. Hence, strong $\calP \otimes \calB(D(A))$-measurability of $\mF\from \Omega\times[0,T]\times D(A) \to D(A)$ follows from strong $\calP \otimes \calB(X)$-measurability as stated in Assumption \ref{ass:WP}. Likewise, strong $\calP \otimes \calB(D(A))$-measurability of $\mG$ can be derived. Lastly, since $R(m,A)$ maps to $D(A)$ and $u_0 \in L_{\calF_0}^p(\Omega;X)$, it holds that $\muo \in L_{\calF_0}^p(\Omega;D(A))$.
        \item First, we recall a folklore result from semigroup theory \cite[Thm.~1.10(iii)]{EngelNagel}: For contraction semigroups, the norm of the resolvent $R(\lambda,A)$ is bounded by $\|R(\lambda,A)\|_{\calL(X)} \le (\TextRe \lambda)^{-1}$ for all $\TextRe \lambda >0$. Hence, $\|mR(m,A)\|_{\calL(X)} \le m \cdot \frac{1}{m} = 1$ is contractive. Together with this observation, Lipschitz continuity of $F$ and $G$ implies uniform Lipschitz continuity of $\mF$ and $\mG$ with the same Lipschitz constant, respectively.
        \item By assumption, $F$ is of linear growth on $X$. Linear growth of $\mF$ on $D(A)$ with $\LFm=(2m+1)\LF$ follows from the identity $AR(m,A)=mR(m,A)-I$, observing that
        \begin{align*}
            \|\mF(\omega,t,x)\|_{D(A)} &= \|mAR(m,A)F(\omega,t,x)\| +\|mR(m,A)F(\omega,t,x)\|\\
            &
            \le \|m[mR(m,A)-I]F(\omega,t,x)\|+\|F(\omega,t,x)\|\\
            &\le (2m+1)\|F(\omega,t,x)\| \le (2m+1)\LF(1+\|x\|).
        \end{align*}
        Linear growth of $\mG$ with $\LGm=(2m+1)\LG$ follows analogously.
        \item It suffices to prove that $mR(m,A) \to I$ in the strong operator topology as $m \to \infty$. Since $A$ is densely defined and closed as the generator of a $C_0$-semigroup, this follows from  \cite[Lemma~3.4]{EngelNagel}. \qedhere
    \end{enumerate}
\end{proof}

For a comprehensive discretisation error analysis, we start by investigating the continuous regularisation error. The following lemma will prove helpful in doing so for $Z \in \{X,\gHX\}$ and with $\psi$ chosen based on the nonlinearity $F$ or noise $G$.

\begin{lemma}
\label{lem:uniformConvSOT}
    Let $Z$ be a Banach space, $\psi\from \Omega \times [0,T] \to Z$ have continuous paths almost surely and assume that 
    \[\sup_{t \in [0,T]} \|\psi(\cdot,t)\|_Z \in L^p(\Omega).\]
    Let $R_n,R \in \calL(Z)$, $n \in \N$, be such that $R_n \to R$ strongly as $n \to \infty$. Then
    \[
    \lim_{n \to \infty} \Big\| \sup_{t \in [0,T]} 
    \| (R_n-R)\psi(\cdot,t)\|_Z\Big\|_p =0.
    \]
\end{lemma}
\begin{proof}
    By continuity of paths of $\psi$, the set $\psi(\omega,[0,T]) \seq Z$ is compact for a.e. $\omega \in \Omega$. Since by assumption $R_n$ converges to $R$ in the strong operator topology, \cite[Proposition A.3]{EngelNagel} yields uniform convergence of $R_n$ to $R$ on compact sets in $Z$ as $n \to \infty$ for a.e. $\omega \in \Omega$. Hence,
    \[
        \sup_{t \in [0,T]} \|(R_n-R)\psi(\omega,t)\|_Z \xrightarrow[n\to\infty]{} 0 \quad \text{ for a.e. }\omega \in \Omega.
    \]
    Due to the assumed integrability of the supremum of $\psi$ in time, the desired statement follows from dominated convergence in $L^p(\Omega)$.
\end{proof}

\begin{lemma}[Convergence of continuous regularisation]
\label{lemma:N1regularisationContinuous}
    Suppose that Assumption \ref{ass:FG_conv} holds for some $p \in [2,\infty)$. Suppose that $A$ generates a $C_0$-contraction semigroup $(S(t))_{t \ge 0}$ on $X$ and let $u_0 \in L_{\calF_0}^p(\Omega;X)$. Denote by $U$ the mild solution of \eqref{eq:StEvolEqnFG} and by $\mU$ the mild solution of \eqref{eq:StEvolEqnDefnU} with $\mF,\mG,\muo$ as defined in \eqref{eq:defFnGnuon} for $m \in \N$. Then 
    \begin{equation}
    \label{eq:boundN1}
            \lim_{m \to \infty} \bigg\lVert\sup_{t \in [0,T]}\|U(t)-\mU(t)\|\bigg\rVert_p =0.
    \end{equation}
\end{lemma}

\begin{proof}
    Let $\li{_m}{V}\ce U- \mU$ and $ \tau \in [0,T]$. Then $\li{_m}{V}$ is given by
    \begin{align*}
        \li{_m}{V}(t) &= S(t)[u_0-\muo]+\int_0^t S(t-s)[F(s,U(s))-\mF(s,\mU(s))]\,\rmd s\\
        &\phantom{= }+ \int_0^t S(t-s)[G(s,U(s))-\mG(s,\mU(s))]\,\rmd W_H(s),
    \end{align*}
    which implies
    \begin{align*}
        \li{_m}E_1(\tau) &\ce \bigg\lVert\sup_{t \in [0,\tau]}\|\li{_m}{V}(t)\|\bigg\rVert_p\\ &\le \bigg\lVert\sup_{t \in [0,\tau]} \|S(t)[u_0-\muo]\|\bigg\rVert_p+\bigg\lVert\sup_{t \in [0,\tau]}\int_0^t\|S(t-s)[F(s,U(s))-\mF(s,\mU(s))]\|\ds\bigg\rVert_p\\
        &\phantom{\ce }+ \bigg\lVert\sup_{t \in [0,\tau]}\bigg\|\int_0^t S(t-s)[G(s,U(s))-\mG(s,\mU(s))]\dWHs\bigg\|\bigg\rVert_p\\
        &\ec E_{1,1}(\tau)+E_{1,2}(\tau)+E_{1,3}(\tau).
    \end{align*}
    We proceed to bound the terms individually.  
    For the initial value term, contractivity of $S$, strong convergence of $mR(m,A)$ to $I$ on $X$ and dominated convergence in $L^p(\Omega)$ yield the existence of $m_0 \in \N$ such that for all $m \ge m_0$, 
    \begin{align}
    \label{eq:N11}
        E_{1,1}(\tau) &\le \|u_0-\muo\|_{L^p(\Omega;X)} = \|[I-mR(m,A)]u_0\|_{L^p(\Omega;X)} <\frac{\varepsilon}{3}.
    \end{align}
    Next, we estimate
    \begin{align*}
        E_{1,2}(\tau) &\le \bigg\|\int_0^\tau \|F(s,U(s))-\mF(s,U(s))\|\ds \bigg\|_p+\CF\bigg\|\int_0^\tau \|U(s)-\mU(s)\|\ds \bigg\|_p\\
        &\le \tau \bigg\|\sup_{s \in [0,\tau]} \|F(s,U(s))-\mF(s,U(s))\|\bigg\|_p + \CF \int_0^\tau \bigg\|\sup_{r \in [0,s]}\|\li{_m}V(r)\|\bigg\|_p\ds 
    \end{align*}
    using contractivity of $S$ and uniform Lipschitz continuity of $\mF$. By Theorem \ref{thm:wellposed}, $U$ almost surely has continuous paths. Combined with continuity of $F$ in time and space as follows from Assumption \ref{ass:FG_conv}, this implies that $\psi\from \Omega \times [0,T] \to X$, $\psi(\omega,t) \ce F(\omega,t,U(t))$ also has continuous paths almost surely. Furthermore, linear growth and Theorem \ref{thm:wellposed} imply
    \begin{align*}
        \bigg\|&\sup_{s \in [0,\tau]}\|F(s,U(s))\|\bigg\|_p \le \LF (1+\|U\|_{L^p(\Omega;C([0,T];X))})<\infty.
    \end{align*}
    Hence, Lemma \ref{lem:uniformConvSOT} applied to $\psi$ on $Z =X$ with $R_n=nR(n,A)$ and $R=I$ yields the existence of $m_1 \in \N$ such that for all $m \ge m_1$,
    \begin{equation*}
        \bigg\|\sup_{s \in [0,\tau]} \|F(s,U(s))-\mF(s,U(s))\|\bigg\|_p < \frac{\varepsilon}{3T}.
    \end{equation*}
    In conclusion, for $m \ge m_1$ the Cauchy-Schwarz inequality yields
    \begin{equation}
    \label{eq:N12}
        E_{1,2}(\tau) < \frac{\varepsilon}{3}+\CF \int_0^\tau \li{_m}E_1(s) \ds \le \frac{\varepsilon}{3}+\CF  T^{1/2} \Big(\int_0^\tau \li{_m}E_1(s)^2\ds\Big)^{1/2}.
    \end{equation}
    
    Via the maximal inequality from Theorem \ref{thm:maxIneqQuasiContractive}, the triangle inequality in $L^p(\Omega;L^2(0,\tau;\gHX))$, uniform Lipschitz continuity of $\mG$, and Fubini's theorem we obtain
    \begin{align*}
        E_{1,3}(\tau) &\le C_{p,D} \Big\lVert\Big(\int_0^\tau \|G(s,U(s))-\mG(s,\mU(s))\|_\gHX^2\ds\Big)^{1/2}\Big\rVert_p\\
        &\le C_{p,D} \Big\lVert\Big(\int_0^\tau \|G(s,U(s))-\mG(s,U(s))\|_\gHX^2\ds\Big)^{1/2}\Big\rVert_p\\
        &\phantom{\le }+ C_{p,D}\CG \Big\lVert\int_0^\tau \|U(s)-\mU(s)\|^2\ds\Big\rVert_{p/2}^{1/2}\\
        &\le C_{p,D}\tau^{1/2} \bigg\lVert \sup_{s\in [0,T]}\big\|G(s,U(s))-\mG(s,U(s))\big\|_\gHX\bigg\rVert_p\\
        &\phantom{\le }+C_{p,D}\CG \bigg( \int_0^\tau\bigg\|\sup_{r \in [0,s]} \|U(r)-\mU(r)\|\bigg\|_p^2\ds\bigg)^{1/2}.
    \end{align*}
    By the left ideal property of $\gHX$, see Proposition \ref{prop:leftIdeal}, the resolvent $R(m,A)$ extends to a linear and bounded operator on $\gHX$ for $m \in \rho(A)$. Hence, arguing as for the nonlinear terms, Lemma \ref{lem:uniformConvSOT} with $Z=\gHX$ and $\psi(\omega,t)=G(\omega,t,U(t))$ yields the existence of $m_2 \in \N$ such that for all $m \ge m_2$,
    \begin{equation*}
        \bigg\lVert \sup_{s\in [0,T]}\|G(s,U(s))-\mG(s,U(s))\|_\gHX\bigg\rVert_p < \frac{\varepsilon}{3C_{p,D}T^{1/2}}.
    \end{equation*}
    Hence, for $m \ge m_2$
    \begin{equation}
    \label{eq:N13}
        E_{1,3}(\tau) \le \frac{\varepsilon}{3} + C_{p,D}\CG\Big( \int_0^\tau \li{_m}E_1(s)^2\ds\Big)^{1/2}.
    \end{equation}
    Altogether, we deduce from \eqref{eq:N11}, \eqref{eq:N12}, and \eqref{eq:N13} that
    \begin{align*}
        \li{_m}E_1(\tau) &\le \varepsilon 
        +\beta_{p,D,T} \Big(\int_0^\tau \li{_m}E_1(s)^2\ds\Big)^{1/2}
    \end{align*}
    with $\beta_{p,D,T}\ce \CF T^{1/2}+C_{p,D}\CG$.
    An application of the continuous version of Gronwall's inequality from Lemma \ref{lem:gronwallvar} yields
    \begin{align*}
        \li{_m}E_1(\tau) \le \varepsilon \cdot(1+\beta_{p,D,T}^2 \tau)^{1/2}\exp\Big(\frac12+\frac12 \beta_{p,D,T}^2 \tau\Big).
    \end{align*}
    The required statement follows by setting $\tau=T$.
\end{proof}

The regularised discrete solution obtained by applying the contractive time discretisation scheme as in \eqref{eq:defUj} is given by the variation-of-constants formula
\begin{equation}
\label{eq:nUjVoC}
    \mU^j \ce R_k^j \muo + k \sum_{i=0}^{j-1} R_k^{j-i}\mF(t_i,\mU^i) + \sum_{i=0}^{j-1} R_k^{j-i}\mG(t_i,\mU^i) \Delta W_{i+1}
\end{equation}
for $j=0,\ldots,N_k$.

The next error investigated is the numerical discretisation error for the regularised problem, where we make use of the fact that $\mF$ and $\mG$ are mapping into spaces with additional regularity. By means of the regularisation, we are now in the position to apply the results from Section \ref{sec:2smoothRates}.

\begin{corollary}[Convergence of regularised discretisation]
\label{cor:regDiscrError}  
    Suppose that Assumption \ref{ass:FG_conv} holds for some $p\in [2,\infty)$ and $\alpha\in (0,1]$. Further, suppose that $A$ generates a $C_0$-contraction semigroup on both $X$ and $D(A)$, and let $u_0 \in L_{\calF_0}^p(\Omega;X)$. Let $(R_k)_{k>0}$ be a time discretisation scheme that is contractive on both $X$ and $D(A)$ and that approximates $S$ to order $\alpha$ on $D(A)$. Let $m \in \N$. Denote by $\mU$ the mild solution of \eqref{eq:StEvolEqnDefnU} with $\mF,\mG,\muo$ as defined in \eqref{eq:defFnGnuon} and by $(\mU^j)_{j=0,\ldots,N_k}$ the temporal approximations as defined in \eqref{eq:nUjVoC}. Then 
    \begin{equation}
    \label{eq:boundN2}
            \lim_{k \to 0}\bigg\lVert\max_{0 \le j \le N_k}\|\mU(t_j)-\mU^j\|\bigg\rVert_p =0.
    \end{equation}
\end{corollary}
\begin{proof}
    First, we note that, as $X$, $D(A)$ is a $2$-smooth Banach space, see Subsection \ref{subsec:2smoothBanach}. Global Lipschitz continuity of $\mF,\mG$ on $X$, $D(A)$-invariance, and linear growth on $D(A)$ as stated in Assumption \ref{ass:FG_structured}\ref{item:FGstruct_Lipschitz}, \ref{item:FGstruct_Yinvariance}, and \ref{item:FGstruct_linearGrowth} were already proven in Proposition \ref{prop:prop}.
    Hölder continuity of $\mF$ and $\mG$ as in Assumption \ref{ass:FG_structured}\ref{item:FGstruct_Hölder} follows immediately from the respective Assumption \ref{ass:FG_conv}\ref{item:FGconv_Hölder} on $F$ and $G$. Lastly, $\muo=mR(m,A)u_0 \in L_{\calF_0}^p(\Omega;D(A))$ due to the regularising property of the resolvent. 
    
    Hence, Theorem \ref{thm:convergenceRate} is applicable to $\mU$ and its discretisation $(\mU^j)_{j=0,\ldots,N_k}$ with $Y=D(A)$, nonlinearity $\mF$, and noise $\mG$. It yields the desired convergence, even with a rate depending on the Hölder continuity in time of $F$ and $G$.
\end{proof}

Note that the convergence of the regularised discretisation is not uniform in the regularisation parameter $m \in \N$. This leads to additional challenges in the proof of the main result, which we now pass to.

\begin{proof}[Proof of Theorem \ref{thm:convergence}]
    Let $\lfloor t \rfloor \ce t_j$ for $t \in [t_j,t_{j+1})$, $0 \le j \le N_k-1$, and $\lfloor T \rfloor \ce T$. Then
    \begin{equation*}
        \bigg\|\sup_{t \in [0,T]} \|U(t)-\tilde{U}(t)\| \bigg\|_p 
        \le \bigg\|\sup_{t \in [0,T]} \|U(t)-U(\lfloor t\rfloor)\| \bigg\|_p +\bigg\|\max_{0 \le j \le N_k} \|U(t_j)-U^j\| \bigg\|_p.
    \end{equation*}
    Theorem \ref{thm:wellposed} implies pathwise continuity of the mild solution $U$. Clearly, $U$ is also uniformly continuous on $[0,T]$, which together with dominated convergence in $L^p(\Omega)$ yields convergence of the first term to $0$ as $k \to 0$. It remains to show convergence of the discretisation error. To this end, let $N \in \{0,\ldots,N_k\}$ and fix some $m \in \N$ to be determined later. We further decompose the discretisation error at the first $N+1$ grid points into three parts
    \begin{align}
    \label{eq:fullerrorBeginning}
        E(N) &\ce \bigg\| \max_{0 \le j \le N} \|U(t_j)-U^j\|\bigg\|_p \nonumber\\
        &\le \bigg\lVert\max_{0 \le j \le N}\|U(t_j)-\mU(t_j)\|\bigg\rVert_p+ \bigg\lVert\max_{0 \le j \le N}\|\mU(t_j)-\mU^j\|\bigg\rVert_p + \bigg\|\max_{0\le j \le N}\|\mU^j-U^j\|\bigg\|_p\nonumber\\
        &\ec \li{_m}E_1(N) + \li{_m}E_2(N) + \li{_m}E_3(N).
    \end{align}
    Note that $\li{_m}E_1(N)\to 0$ uniformly in $N$ as $m \to \infty$ as a consequence of Lemma \ref{lemma:N1regularisationContinuous}. Moreover, $\li{_m}E_2(N_k)\to 0$ as $k \to 0$ follows from Corollary \ref{cor:regDiscrError}. It remains to bound the remaining term $\li{_m}E_3(N)$. This will be done in terms of $\li{_m}E_1(N_k)$ and $\li{_m}E_2(N_k)$, which converge in the desired manner, and $E(i)$, $0 \le i \le N-1$, which is dealt with using a Gronwall argument as illustrated in Step 1. The bound itself will be obtained in Step 2 of the proof. 

    \textit{Claim.} Let $\ve >0$. We claim that there exist $m_0=m_0(\ve) \in \N$ and $C=C(p,D,F,G,T)\ge 0$ such that for $m \ge m_0$,
    \begin{align}
    \label{eq:claimStep1}
        \li{_m}E_3(N) \le \frac{\varepsilon}{2} + C[ \li{_m}E_1(N_k)+ \li{_m}E_2(N_k)]+ C\bigg(k \sum_{i=0}^{N-1} E(i)^2\bigg)^{1/2}.
    \end{align}
    \textit{Step 1.} 
    We show that the claim suffices to prove the convergence of $E(N_k)$ as $k \to 0$. Indeed, noting that $\li{_m}E_i(N) \le \li{_m}E_i(N_k)$ for $i=1,2$, we conclude from \eqref{eq:fullerrorBeginning} that
    \begin{equation*}
        E(N) \le \frac{\varepsilon}{2} + (C+1)[\li{_m}E_1(N_k)+\li{_m}E_2(N_k)]+ C\bigg(k \sum_{i=0}^{N-1} E(i)^2\bigg)^{1/2}.
    \end{equation*}
    An application of Gronwall's inequality from Lemma \ref{lem:KruseGronwall} results in
    \begin{equation*}
        E(N) \le \Big(\frac{\varepsilon}{2} + (C+1)[\li{_m}E_1(N_k)+\li{_m}E_2(N_k)]\Big)(1+C^2t_N)^{1/2}\exp\bigg(\frac{1+C^2t_N}{2}\bigg)
    \end{equation*}
    for all $m \ge m_0$. By Lemma \ref{lemma:N1regularisationContinuous}, there exists $m_1 \in \N$ such that 
    \begin{equation*}
        \li{_m}E_1(N_k) \le \bigg\| \sup_{t \in [0,T]} \|U(t)-\mU(t)\| \bigg\|_p \le \frac{\varepsilon}{2(C+1)}
    \end{equation*}
    for all $m \ge m_1$ and $N_k \in \N$. Fix some  $m \ge \max\{m_0,m_1\}$. Then
    \begin{equation*}
        E(N) \le \big(\varepsilon + (C+1)\li{_m}E_2(N_k)\big)(1+C^2T)^{1/2}\exp\bigg(\frac{1+C^2T}{2}\bigg).
    \end{equation*}
    Corollary \ref{cor:regDiscrError} gives $\li{_m}E_2(N_k) \to 0$ as $N_k \to \infty$ or, equivalently, $k \to 0$. Since $\ve>0$ was chosen arbitrarily, we conclude $E(N_k) \to 0$ as $k \to 0$, which proves the desired convergence statement.
    
    \textit{Step 2.} We proceed to prove the claim \eqref{eq:claimStep1} from Step $1$. The error can be divided into an initial value part, a nonlinear part, and a noise part according to
    \begin{align}
    \label{eq:N3split}
       \li{_m}E_3(N) &= \bigg\|\max_{0\le j \le N}\|\mU^j-U^j\|\bigg\|_p \nonumber\\
        &\le \bigg\|\max_{0\le j \le N}\|R_k^j(\muo-u_0)\|\bigg\|_p+\bigg\| \max_{0\le j \le N} \bigg\| k \sum_{i=0}^{j-1} R_k^{j-i}[F(t_i,U^i)-\mF(t_i,\mU^i)]\bigg\|\bigg\|_p\nonumber\\
        &\phantom{= }+ \bigg\| \max_{0\le j \le N} \bigg\|\sum_{i=0}^{j-1} R_k^{j-i}[G(t_i,U^i)-\mG(t_i,\mU^i)]\Delta W_{i+1}\bigg\|\bigg\|_p\nonumber\\
        &\ec E_{3,1}+E_{3,2}+E_{3,3},
    \end{align}
    where the dependence $E_{3,\ell}=E_{3,\ell}(N,m,k)$ for $\ell=1,2,3$ is omitted in the notation. We bound all three terms individually. First, we observe that by contractivity of $R_k$ and pointwise convergence of $mR(m,A) \to I$, there exists $m_2 \in \N$ such that for all $m \ge m_2$
    \begin{align}
    \label{eq:N31}
        E_{3,1} \le \|\muo-u_0\|_{L^p(\Omega;X)} = \|(mR(m,A)-I)u_0\|_{L^p(\Omega;X)}<\frac{\varepsilon}{6}.
    \end{align}
    Second, we consider the nonlinear part of the error. For $0 \le j \le N$ and $0 \le i \le j-1$, we estimate 
    \begin{align}
    \label{eq:splitNLtermsFFn}
        \|&R_k^{j-i}[F(t_i,U^i)-\mF(t_i,\mU^i)]\|
         \nonumber\\
         & \le \|F(t_i,U^i)-F(t_i,U(t_i))\|+\|F(t_i,U(t_i))-\mF(t_i,U(t_i))\|\nonumber\\
         &\phantom{\le }+ \|\mF(t_i,U(t_i))-\mF(t_i,\mU(t_i))\| + \|\mF(t_i,\mU(t_i))-\mF(t_i,\mU^i)\|\nonumber\\
         &\le \CF \|U(t_i)-U^i\| +\|F(t_i,U(t_i))-\mF(t_i,U(t_i))\|\nonumber\\
         &\phantom{\le }+\CF \|U(t_i)-\mU(t_i)\|+\CF \|\mU(t_i)-\mU^i\|,
    \end{align} 
    where in the last step we have used uniform Lipschitz continuity of $\mF$ and $F$. The reason for splitting the error this way is that the difference between $F$ and its regularised counterpart $\mF$ is evaluated in the values $U(t_i)$ of the mild solution at the time grid points. Since the mild solution has continuous paths, this enables us to apply Lemma \ref{lem:uniformConvSOT} as seen in the proof of Lemma \ref{lemma:N1regularisationContinuous}. This yields uniform convergence, in particular uniformly in the number of time steps. 
    Summing over $i$, multiplying by $k$, taking the maximum over all $j$ and taking norms in $L^p(\Omega)$, we conclude from Minkowski's inequality in $L^p(\Omega)$ that
    \begin{align*}
        E_{3,2} &= \bigg\| \max_{0\le j \le N} \bigg\| k \sum_{i=0}^{j-1} R_k^{j-i}[F(t_i,U^i)-\mF(t_i,\mU^i)]\bigg\|\bigg\|_p\\
        &\le \CF\bigg\| k \sum_{i=0}^{N-1} \|U(t_i)-U^i\|\bigg\|_p + \CF\bigg\|  k \sum_{i=0}^{N-1} \|U(t_i)-\mU(t_i)\|\bigg\|_p\\
        &\phantom{\le }+ \CF\bigg\|  k \sum_{i=0}^{N-1} \|\mU(t_i)-\mU^i\|\bigg\|_p+\bigg\|  k \sum_{i=0}^{N-1} \|F(t_i,U(t_i))-\mF(t_i,U(t_i))\|\bigg\|_p\\
        &\le \CF k \sum_{i=0}^{N-1} E(i) + \CF T [ \li{_m}E_1(N_k)+   \li{_m}E_2(N_k)]+ T \bigg\| \sup_{t \in [0,T]} \|F(t,U(t))-\mF(t,U(t))\|\bigg\|_p. 
    \end{align*}
    As demonstrated in detail in the proof of Lemma \ref{lemma:N1regularisationContinuous}, Lemma \ref{lem:uniformConvSOT} yields the existence of $m_3 \in \N$ such that for all $m \ge m_3$,
    \begin{equation*}
        \bigg\| \sup_{t \in [0,T]} \|F(t,U(t))-\mF(t,U(t))\|\bigg\|_p < \frac{\varepsilon}{6T}.
    \end{equation*}
    Consequently, from Cauchy-Schwarz's inequality it follows that for all $m \ge m_3$,
    \begin{equation}
    \label{eq:N32}
        E_{3,2}\le \CF\sqrt{T}\Big(k\sum_{i=0}^{N-1} E(i)^2\Big)^{1/2} +  \CF T [\li{_m}E_1(N_k)+  \li{_m}E_2(N_k)] + \frac{\varepsilon}{6}.
    \end{equation}
    To bound the last term $E_{3,3}$ in \eqref{eq:N3split}, we apply Lemma \ref{lem:PinelisBound} with $Q_i\ce G(t_i,U^i)-\mG(t_i,\mU^i)$.
    This yields
    \begin{align*}
        E_{3,3}
        &\le B_{p,D} \Big(k \sum_{i=0}^{N-1} \big\| \|[G(t_i,U^i)-\mG(t_i,\mU^i)]\|_\gHX\big\|_p^2\Big)^{1/2}
    \end{align*}
    with $B_{p,D} \ce 10D\sqrt{p}(\frac{10p^2}{p-1}+1)$ recalling that $C_{p,D}=10D\sqrt{p}$.
    As for the nonlinear terms in \eqref{eq:splitNLtermsFFn}, we split the term $\|G(t_i,U^i)-\mG(t_i,\mU^i)\|_\gHX$ in such a way that the difference of $G$ and $\mG$ is evaluated at $U(t_i)$ rather than the discrete approximations $U^i$ or $\mU^i$. After an application of the triangle inequality in $\ell^2(\{0,\ldots,N-1\};L^p(\Omega;\gHX))$ this results in
    \begin{align*}
        E_{3,3}
        &\le B_{p,D}\bigg[\CG \Big(k \sum_{i=0}^{N-1}\| U(t_i)-U^i \|_{L^p(\Omega;X)}^2\Big)^{1/2}\\
        &\phantom{\le }+\CG\Big(k \sum_{i=0}^{N-1}\| U(t_i)-\mU(t_i)\|_{L^p(\Omega;X)}^2\Big)^{1/2} +\CG\Big(k \sum_{i=0}^{N-1}\| \mU(t_i)-\mU^i\|_{L^p(\Omega;X)}^2\Big)^{1/2}\\
        &\phantom{\le }+ \Big(k \sum_{i=0}^{N-1}\big\|\| G(t_i,U(t_i))-\mG(t_i,U(t_i))\|_\gHX \big\|_p^2\Big)^{1/2}\bigg]\\
        &\le B_{p,D}\bigg[\CG \Big(k \sum_{i=0}^{N-1}E(i)^2\Big)^{1/2}+\CG\sqrt{T}[\li{_m}E_1(N_k)+\li{_m}E_2(N_k)]\\
        &\phantom{\le }+ \sqrt{T}\bigg\|\sup_{t \in [0,T]}\| G(t,U(t))-\mG(t,U(t))\|_\gHX \bigg\|_p\bigg].
    \end{align*}
    We recall from the proof of Lemma \ref{lemma:N1regularisationContinuous} that the left ideal property of $\gHX$ allows us to apply Lemma \ref{lem:uniformConvSOT} on $Z=\gHX$. We infer that there is $m_4 \in \N$ such that for all $m \ge m_4$ the bound
    \begin{equation*}
        \bigg\| \sup_{t \in [0,T]} \|G(t,U(t))-\mG(t,U(t))\|\bigg\|_p < \frac{\varepsilon}{6B_{p,D}\sqrt{T}}
    \end{equation*}
    holds. Thus, for $m \ge m_4$,
    \begin{align}
    \label{eq:N33}
        E_{3,3}
        &\le B_{p,D}\CG \bigg(k \sum_{i=0}^{N-1}E(i)^2\bigg)^{1/2}+B_{p,D}\CG\sqrt{T}[\li{_m}E_1(N_k)+\li{_m}E_2(N_k)]+\frac{\varepsilon}{6}.
    \end{align}
    Inserting the bounds \eqref{eq:N31}, \eqref{eq:N32}, and \eqref{eq:N33} into \eqref{eq:N3split} proves the claim \eqref{eq:claimStep1} with $C \ce \max\{\sqrt{T},1\} \cdot$ $(\CF\sqrt{T}+B_{p,D}\CG)$ and $m_0 \ce \max\{m_2,m_3,m_4\}$. This finishes the proof.
\end{proof}

\section{Application to the stochastic Schrödinger equation}
\label{sec:example}

To illustrate the convergence results from Section \ref{sec:convergence}, we consider the stochastic Schrödinger equation with a potential and linear multiplicative noise
\begin{align}
\label{eq:linearSchroedingerPotentialMultNoise}
    \Bigg\{\begin{split} \rmd u &= -\iu(\Delta + V) u \;\rmd t-\iu u\; \rmd W~~~ \text{ on }[0,T],\\
    u(0)&=u_0
    \end{split}
\end{align}
and its nonlinear variant with $\phi\from \C\to \C$ and $\psi\from \C \to \C$,
\begin{align}
\label{eq:SchrödingerNL}
    \Bigg\{\begin{split} \rmd u &= -\iu(\Delta u + V u + \phi(u))\;\rmd t-\iu \psi(u)\; \rmd W~~~ \text{ on }[0,T],\\
    u(0)&=u_0
    \end{split}
\end{align}
in $\R^d$ for $d \in \N$. Here, $\{W(t)\}_{t\ge 0}$ is a square integrable $\K$-valued $Q$-Wiener process, $\K\in \{\R,\C\}$, with respect to a normal filtration $(\filtrF_t)_{t \ge 0}$, $V$ is a $\K$-valued potential, $u_0$ is an $\filtrF_0$-measurable random variable, and $T>0$. Pathwise uniform convergence of contractive time discretisation schemes is known for this equation for sufficiently regular $V$, $Q$, and $u_0$, and convergence rates are at hand (\cite[Thm.~6.12, 6.13]{KliobaVeraar23}, \cite[Thm.~5.5]{AC18}). We aim at relaxing the regularity conditions imposed on the potential $V$ as well as the covariance operator $Q$ and allowing for rough initial data $u_0$, while maintaining pathwise uniform convergence.

Let $\sigma \ge 0$ and write $L^2=L^2(\R^d)$, $L^\infty=L^\infty(\R^d)$, and $H^\sigma=H^\sigma(\R^d)$. We will also be using the Bessel potential spaces $H^{\sigma,q}=H^{\sigma,q}(\R^d)$, which coincide with the classical Sobolev spaces $W^{\sigma,q}(\R^d)$ if $\sigma\in \N$ and $q\in (1, \infty)$. For details on these spaces, we refer the interested reader to \cite{bergh2012interpolationBesselPotential,Tr1}.

We are concerned with covariance operators $Q\in \calL(L^2)$ of trace class. More precisely, we assume that for an orthonormal basis $(h_n)_{n \in \N}$ of $L^2$, the covariance operator decomposes as
\begin{equation}
\label{eq:QtraceClass}
    Q=\sum_{n \in \N} \lambda_n (h_n \otimes h_n)\quad\text{ with }\quad \sum_{n \in \N} \lambda_n =C_\lambda<\infty,\quad \sup_{n \in \N} \big(\|h_n\|_{L^\infty}+\|h_n\|_{H^{\sigma,d/\sigma}}\big) < \infty
\end{equation}
for some constant $C_\lambda\ge 0$. For $\sigma=0$, $H^{\sigma,d/\sigma}$ should be interpreted as $L^\infty$. The conditions \eqref{eq:QtraceClass} are equivalent to $Q^{1/2} \in \calL(L^2,L^\infty\cap H^{\sigma,d/\sigma})$. While the last condition constitutes an additional regularity assumption on $Q$, still a wide range of operators is covered due to the Sobolev index of $H^{\sigma,d/\sigma}$ being $0$. In particular, $H^{\sigma,d/\sigma}$-regularity does not result in any Hölder regularity, not even continuity.

The following theorem on the linear Schrödinger equation covers among others the case $\sigma=0$, for which $H^{\sigma,d/\sigma}$ should be interpreted as $L^\infty$. More general nonlinearities can be treated when restricting considerations to only the case $\sigma=0$, see Theorem \ref{thm:NLSchrödinger}.

\begin{theorem}
\label{thm:linearSchrödinger}
    Let $d \in \N$, $\sigma \in [0,\frac{d}{2})$, and $p \in [2,\infty)$. Assume that $V \in L^\infty\cap H^{\sigma,\frac{d}{\sigma}}$ and $u_0 \in L_{\calF_0}^p(\Omega;H^\sigma)$. Suppose that the covariance operator $Q \in \calL(L^2)$ satisfies \eqref{eq:QtraceClass}. Let $(R_{k})_{k>0}$ be a time discretisation scheme that is contractive on $H^\sigma$ and $H^{\sigma+2}$. Assume $R$ approximates $S$ to some order $\alpha \in (0,1]$ on $H^{\sigma+2}$. Denote by $U$ the mild solution of the linear stochastic Schrödinger equation with multiplicative noise \eqref{eq:linearSchroedingerPotentialMultNoise}
    and by $(U^j)_{j=0,\ldots,N_k}$ the temporal approximations as defined in \eqref{eq:defUj}.   
    Define the piecewise constant extension $\tilde{U}:[0,T] \to L^p(\Omega;X)$ by $\tilde{U}(t) \ce U^j$ for $t \in [t_j,t_{j+1})$, $0 \le j \le N_k-1$, and $\tilde{U}(T) \ce U^{N_k}$. Then
    \begin{align}
        \lim_{k \to 0} \bigg\| \sup_{t \in [0,T]} \|U(t)-\tilde{U}(t)\|_{H^\sigma}\bigg\|_p =0.
    \end{align}
\end{theorem}
\begin{proof}
    Let $X \ce H^\sigma$. The semigroup generated by $A= -\iu \Delta$ is contractive on both $X$ and $D(A)=H^{\sigma+2}$ \cite[Lemma~2.1]{AC18}. We claim that Assumption \ref{ass:FG_conv} is satisfied for $F(\omega,t,u) \ce -\iu V u$ and $G(\omega,t,u) \ce -\iu M_u Q^{1/2}$ with $M_u$ denoting the multiplication operator associated to $u$. At first, we show Lipschitz continuity of $F$ on $X$. Let $q_1 = \frac{2d}{d-2\sigma}$ and $q_2=\frac{d}{\sigma}$. Then $\frac{1}{q_1}+\frac{1}{q_2}=\frac{1}{2}$ and $q_1<\infty$ because $d>2\sigma$. By classical Sobolev and Bessel potential space embeddings \cite[Thm.~6.5.1]{bergh2012interpolationBesselPotential}, $H^\sigma$ embeds into $L^{q_1}$. Hence, an application of the product estimate \cite[Prop.~2.1.1]{TaylorBook} yields
    \begin{align}
    \label{eq:applyTaylor}
        \|F(u)\|_{H^\sigma}&=\|V\cdot u\|_{H^\sigma} \lesssim \|V\|_{H^{\sigma,q_2}}\|u\|_{L^{q_1}} + \|V\|_{L^\infty}\|u\|_{H^\sigma}\lesssim (\|V\|_{H^{\sigma,d/\sigma}}+\|V\|_{L^\infty})\|u\|_{H^\sigma}
    \end{align}
    for $u \in H^\sigma$, i.e., linear growth of $F$. Lipschitz continuity of $F$ follows from the above considerations, noting that $F$ is linear.
    
    Next, Lipschitz continuity and linear growth of $G$ are to be derived from the trace class condition of $Q$. Set $H=L^2$ and let $Q=\sum_{n \in \N} \lambda_n (h_n \otimes h_n)$ with $(h_n)_{n\in \N}$ and $\lambda_n$ as in \eqref{eq:QtraceClass}. Since $H^\sigma$ is a Hilbert space, it suffices to consider Hilbert-Schmidt norms. Using the product estimate from \eqref{eq:applyTaylor} in the inequality marked with $(*)$, we calculate
    \begin{align*}
        \|G(u)\|_{\calL_2(L^2,H^\sigma)}^2 &= \|M_u Q^{1/2}\|_{\calL_2(L^2,H^\sigma)}^2 = \sum_{n \in \N} \|u Q^{1/2}h_n\|_{H^\sigma}^2 = \sum_{n \in \N} \lambda_n \|u \cdot h_n\|_{H^\sigma}^2\\
        &\overset{(*)}{\lesssim} \sum_{n \in \N} \lambda_n (\|h_n\|_{H^{\sigma,d/\sigma}}+\|h_n\|_{L^\infty})^2\|u\|_{H^\sigma}^2 \\
        &\le C_\lambda \sup_{n \in \N} \big(\|h_n\|_{H^{\sigma,d/\sigma}}+\|h_n\|_{L^\infty}\big)^2 \|u\|_{H^\sigma}^2.
    \end{align*}
    Linearity of $G$ yields Lipschitz continuity of $G$. In conclusion, Assumption \ref{ass:FG_conv} is satisfied.
    Hence, Theorem \ref{thm:convergence} is applicable and yields the desired convergence statement.
\end{proof}

Note that the convergence can be arbitrarily slow. More precisely, in the general case, there is no $\alpha >0$ such that 
\[ \bigg\| \sup_{t \in [0,T]} \|U(t)-\tilde{U}(t)\|_{H^\sigma}\bigg\|_p \le k^\alpha.\]
Previous results yielding a convergence rate \cite[Theorem~6.3]{KliobaVeraar23} are not applicable in the setting of Theorem \ref{thm:linearSchrödinger}. Clearly, rough initial data prohibits an application of quantified results. However, even for smooth initial data, a convergence rate is out of reach due to the lack of regularity of the potential $V$ and the covariance $Q$. 
Since the embedding $H^\sigma \hra L^{\frac{2d}{d-2\sigma}}$ and the product estimate \cite[Prop.~2.1.1]{TaylorBook} are sharp, there is no $\tilde{\sigma}>\sigma$ such that $F$ or $G$ are mappings of linear growth on $H^{\tilde{\sigma}}$. Consequently, the smoother space required for a convergence rate cannot be found in the setting of this section.

To cover proper nonlinearities as in \eqref{eq:SchrödingerNL}, estimates of the form
\begin{equation}
\label{eq:LipschitzHsigma}
    \|\psi(u)-\psi(v)\|_{H^\sigma} \lesssim \|u-v\|_{H^\sigma}\quad (u,v \in H^\sigma)
\end{equation}
are required to show Lipschitz continuity of $G$. Estimates of this kind are out of reach for $\sigma >0$ in the general case, see \cite[p.~30]{KliobaVeraar23}. In particular, Nemytskij maps are not Lipschitz on $H^\sigma$ for any $\sigma>0$. 

\begin{theorem}
\label{thm:NLSchrödinger}
    Let $d \in \N$ and $p \in [2,\infty)$. Assume that $V \in L^\infty$  and $u_0 \in L_{\calF_0}^p(\Omega;L^2)$. Suppose that the covariance operator $Q \in \calL(L^2)$ satisfies \eqref{eq:QtraceClass}. Let $(R_{k})_{k>0}$ be a time discretisation scheme that is contractive on $L^2$ and $H^2$. Assume $R$ approximates $S$ to some order $\alpha \in (0,1]$ on $H^2$. Let $\phi,\psi\from \C\to \C$ be Lipschitz continuous and such that $\phi(0) = \psi(0) = 0$. Denote by $U$ the mild solution of the nonlinear stochastic Schrödinger equation with multiplicative noise \eqref{eq:SchrödingerNL}
    and by $(U^j)_{j=0,\ldots,N_k}$ the temporal approximations as defined in \eqref{eq:defUj}.   
    Define the piecewise constant extension $\tilde{U}\from [0,T] \to L^p(\Omega;X)$ by $\tilde{U}(t) \ce U^j$ for $t \in [t_j,t_{j+1})$, $0 \le j \le N_k-1$, and $\tilde{U}(T) \ce U^{N_k}$. Then
    \begin{align}
        \lim_{k \to 0} \bigg\| \sup_{t \in [0,T]} \|U(t)-\tilde{U}(t)\|_{L^2}\bigg\|_p =0.
    \end{align}
\end{theorem}

Naturally, the result extends to non-vanishing $\sigma$ in specific cases where Lipschitz continuity on $H^\sigma$ is known. 

\begin{proof}
    We show that Theorem \ref{thm:convergence} is applicable with $F(\omega,t,u) \ce -\iu (V u+\phi(u))$ and $G(\omega,t,u) \ce -\iu M_{\psi(u)} Q^{1/2}$ on $X=L^2$. Analogously to the proof of Theorem \ref{thm:linearSchrödinger}, we obtain the bound
    \begin{equation*}
        \|G(u)-G(w)\|_{\calL_2(L^2,L^2)} \lesssim \sqrt{2C_\lambda} \sup_{n \in \N} \|h_n\|_{L^\infty} \|\psi(u)-\psi(w)\|_{L^2} \le \sqrt{2C_\lambda} C_\psi \sup_{n \in \N} \|h_n\|_{L^\infty} \|u-w\|_{L^2}
    \end{equation*}
    for $u,w \in L^2$, from which we can deduce Lipschitz continuity of $G$.  
    Linear growth of $G$ follows from $G(0)=0$.
    In the same way, one can see that $F(u) = -\iu(Vu+\phi(u))$ is Lipschitz and of linear growth on $L^2$. The statement directly follows from Theorem \ref{thm:convergence}.
\end{proof}

\subsubsection*{Data availability statement} Data sharing is not applicable to this article, as no datasets were created or analysed as part of the current study.

\subsubsection*{Conflict of interest statement} The second author is supported by the VICI subsidy VI.C.212.027 of the Netherlands Organisation for Scientific Research (NWO). The first author received funding by the German Academic Exchange Service (DAAD) to conduct research at Delft University of Technology for six months in 2022. The authors have no further conflicts of interest to declare.


\end{document}